\documentclass[oneside,12pt]{article}

\usepackage{amsmath}
\usepackage{amssymb}
\usepackage{pxfonts}
\usepackage{graphicx}
\usepackage{eucal}
\usepackage{mathrsfs}
\usepackage{theorem}
\usepackage{pifont}
\usepackage{color}
\usepackage{enumitem}
\usepackage{accents}
\usepackage[margin=2.5cm]{geometry}
\usepackage[backref=page]{hyperref}
\usepackage[normalem]{ulem}
\usepackage{dsfont}
\usepackage{bbding}
\usepackage{tikz}
\usepackage{tikz-3dplot}
\usetikzlibrary{patterns}
\usetikzlibrary{decorations.pathreplacing, decorations.markings, shapes.symbols}

\definecolor{shadecolor}{rgb}{0.8,0.8,0.8}
\definecolor{ocre}{cmyk}{0,0.58,0.897,0.0471} 
\definecolor{warmblue}{cmyk}{0.667,0.333,0,0.4} 

\usepackage{tocloft}

\include{./MyMacros}

\theoremheaderfont{\bfseries\rmfamily\upshape}
\newtheorem{theorem}{Theorem}[section]
\newtheorem{lemma}[theorem]{Lemma}
\newtheorem{proposition}[theorem]{Proposition}
\newtheorem{corollary}[theorem]{Corollary}
\newtheorem{definition}[theorem]{Definition}

\newenvironment{proof}{{\flushleft \emph{Proof}:}}{\hfill\ding{110}}
\newenvironment{remark}{{\flushleft\bfseries Remark:}}{}
\newenvironment{comment}{{\flushleft\bfseries Comment:}}{}
\newenvironment{comments}{{\flushleft\bfseries Comments:}}{}

\def\Xint#1{\mathchoice
   {\XXint\displaystyle\textstyle{#1}}%
   {\XXint\textstyle\scriptstyle{#1}}%
   {\XXint\scriptstyle\scriptscriptstyle{#1}}%
   {\XXint\scriptscriptstyle\scriptscriptstyle{#1}}%
   \!\int}
\def\XXint#1#2#3{{\setbox0=\hbox{$#1{#2#3}{\int}$}
     \vcenter{\hbox{$#2#3$}}\kern-.5\wd0}}

\def\dashint{\Xint-}

\usepackage{framed}
\usepackage{enumitem}
\definecolor{shadecolor}{rgb}{0.90,0.90,0.90}

\newcommand{\ind}{\mathds{1}}

\newcommand{\M}{\mathcal{M}}
\newcommand{\Wlami}{{\W_\lambda^j}}
\newcommand{\Glam}{{\Gamma_\lambda}}
\newcommand{\Glami}{{\Gamma_\lambda^j}}
\newcommand{\Sph}{\mathbb{S}^2}
\newcommand{\SphG}{\Sph\setminus\Gamma}

\newcommand{\Circ}{\mathbb{S}^1}

\newcommand{\g}{\frak{g}}
\newcommand{\n}{\mathcal{N}}

\newcommand{\Imm}{\operatorname{Imm}}

\renewcommand{\span}{\operatorname{span}}

\newcommand{\VMO}{\operatorname{VMO}}
\newcommand{\Len}{\operatorname{Length}}
\renewcommand{\div}{\operatorname{div}}

\newcommand{\W}{\Omega}

\newcommand{\widebar}[1]{\mkern1.5mu\overline{\mkern-1.5mu#1\mkern-1.5mu}\mkern1.5mu}
\newcommand{\bW}{\widebar{\W}}
\newcommand{\bD}{\widebar{D}}
\newcommand{\bU}{\widebar{U}}

\newcommand{\dW}{{\partial\W}}

\newcommand{\dD}{{\partial D}}

\newcommand{\kg}{\kappa_g}

\newcommand{\tkg}{\tilde{\kappa}_g}

\newcommand{\tW}{\tilde{\W}}
\newcommand{\talpha}{\tilde{\alpha}}
\newcommand{\tsigma}{\tilde{\sigma}}
\newcommand{\tGamma}{\tilde{\Gamma}}

\newcommand{\G}{\mathcal{G}}

\newcommand{\ip}[1]{\left\langle #1 \right\rangle}
\newcommand{\euc}{\mathfrak{e}}

\newcommand{\Emph}[1]{{\bfseries #1}}

\newcommand{\ca}{c_\alpha}

\newcommand{\Vol}{{\operatorname{Vol}}}
\newcommand{\VolGen}{\operatorname{dVol}}
\newcommand{\VolG}{\VolGen_\g}
\newcommand{\VolGn}{\VolGen_{\g_n}}
\newcommand{\VolGeps}{\VolGen_{\g_\e}}
\newcommand{\iVolG}{\VolGen_{\iota^* \g}}
\newcommand{\iVolGeps}{\VolGen_{\iota^* {\g_\e}}}

\newcommand{\VolSph}{\VolGen_{\Sph}}

\newcommand{\Hom}{{\operatorname{Hom}}}
\newcommand{\SO}{{\operatorname{SO}}}

\newcommand{\Textand}{\qquad\text{ and }\qquad}

\newcommand{\dotHnorm}{\dot{H}^1(\W,\g)}
\newcommand{\dotHnormn}{\dot{H}^1(\W,\g_n)}

\newcommand{\LoneMC}{L^1_\text{MC}(\W,\g)}
\newcommand{\HoneMC}{H^1_\text{MC}(\W,\g)}
\newcommand{\HMoneMC}{\Hminusone_\text{MC}(\W,\g)}

\newcommand{\HthreehalvesImm}{H_{\Imm}^{3/2}}
\newcommand{\HtwoImm}{H_{\Imm}^2}
\newcommand{\Hthreehalves}{H^{3/2}}
\newcommand{\Htwo}{H^2}
\newcommand{\Hhalf}{H^{1/2}}
\newcommand{\Hminushalf}{H^{-1/2}}
\newcommand{\Hone}{H^{1}}
\newcommand{\Hminusone}{H^{-1}}
\newcommand{\Hs}{H^{s}}
\newcommand{\dHs}{\dot{H}^{s}}

\newcommand{\R}{\mathbb{R}}

\newcommand{\tgamma}{\tilde{\gamma}}

\newcommand{\id}{{\operatorname{Id}}}

\newcommand{\pl}{\partial}

\newcommand{\inj}{\hookrightarrow}
\newcommand{\weakly}{\rightharpoonup}

\newcommand{\EqNo}[2]{\stackrel{\eqref{#2}}{#1}}
\newcommand{\Kbnd}{\bbK_{p}}

\newcommand{\brk}[1]{\left(#1\right)}          

\newcommand{\Abs}[1]{\left| #1 \right|}        

\newcommand{\beq}{\begin{equation}}
\newcommand{\eeq}{\end{equation}}
\newcommand{\bsplit}{\begin{split}}
\newcommand{\esplit}{\end{split}}
\newcommand{\baligned}{\begin{aligned}}
\newcommand{\ealigned}{\end{aligned}}

\newcommand{\secref}[1]{Section~\ref{#1}}
\newcommand{\figref}[1]{Figure~\ref{#1}}
\newcommand{\thmref}[1]{Theorem~\ref{#1}}
\newcommand{\defref}[1]{Definition~\ref{#1}}

\newcommand{\propref}[1]{Proposition~\ref{#1}}
\newcommand{\lemref}[1]{Lemma~\ref{#1}}

\newcommand{\deriv}[2]{\frac{d#1}{d#2}}

\newcommand{\pd}[2]{\frac{\partial#1}{\partial#2}}

\newcommand{\limn}{\lim_{n\to\infty}}

\newcommand{\liminfn}{\liminf_{n\to\infty}}

\newcommand{\limsupn}{\limsup_{n\to\infty}}

\newcommand{\frakt}{\mathfrak{t}}
\newcommand{\frakn}{\mathfrak{n}}
\newcommand{\vp}{\varphi}
\newcommand{\bbZ}{\mathbb{Z}}
\newcommand{\calL}{\mathcal{L}}
\newcommand{\bbN}{\mathbb{N}}
\newcommand{\e}{\varepsilon}
\newcommand{\bbK}{\mathbb{K}}

\usepackage[all]{xy}
\usepackage{tikz} 
\usepackage{tkz-euclide}

\newcommand{\btkz}{\begin{tikzpicture}}
\newcommand{\etkz}{\end{tikzpicture}}

\numberwithin{equation}{section}

\begin{document}
\title{The Willmore energy and curvature concentration}
\author{Raz Kupferman*, Cy Maor* and David Padilla-Garza\footnote{Einstein Institute of Mathematics, Hebrew University of Jerusalem}}
\date{}

\maketitle

\begin{abstract}
\noindent We study isometric immersions of a Riemannian surface $(\W,\g)$, where $\W\subset \R^2$, into $\R^3$.
We consider their bending energy,  i.e., the square of the $L^2$-norm of their second fundamental form, which is equivalent to the Willmore functional. 
We obtain two new lower bounds for this energy, one in terms of the Gaussian curvature of the surface, and the other 
in terms of a Burgers vector---a measure of non-flatness connected to torsion. 
These new estimates provide optimal blowup rates of the energy when the curvature is concentrated (e.g., in a conical geometry).
In the more subtle case of dipoles of concentrated curvature,  we use the Burgers vector estimates to obtain an optimal blowup rate in terms of the size of the system.
Our motivation comes from non-Euclidean elasticity, 
in which cones and curvature-dipoles play a central role. 
The lower bounds derived in this work directly yield lower bounds for the elastic energy of thin elastic sheets.
The derivation of the curvature-based lower bound involves an isoperimetric inequality for framed loops, which we believe to be of independent interest.
\end{abstract}

\begingroup
\footnotesize
\tableofcontents
\endgroup

\section{Introduction and main results}

This paper is concerned with lower bounds for the \Emph{bending energy} of an immersed surface $S\subset \R^3$:
\beq
\label{eq:bending_energy}
E_B(S) = \int_S |d\n|_\g^2\, \VolG,
\eeq
where $\n:S\to\Sph$ is the unit normal (the Gauss map) of $S$, $\g$ is its induced metric and $\VolG$ is the corresponding volume form.
The integrand is also equal to the square of the norm of the second fundamental form.
Since 
\[
|d\n|_\g^2 = 4H^2 - 2K,
\] 
where $H$ is the mean curvature and $K$ is the Gaussian curvature, this energy is equivalent to the \Emph{Willmore functional},
\[
W(S) = \int_S |H|^2\,\VolG,
\]
when considering immersions of a given closed manifold (in which case the integral of $K$ is a topological invariant), or when considering isometric immersions of a given (generally non-closed) Riemannian surface $(\W,\g)$.

These energy functionals have been studied extensively in relation to the Willmore conjecture \cite{MN14b}, to its gradient flow \cite{KS02}, and to its critical points \cite{Riv08}, to name a few directions.
Our initial motivation in studying this functional comes from material science, where the bending energy $E_B$ plays a prominent role in the elastic theory of thin sheets; see Section~\ref{sec:literature} for more details.

A natural problem is to establish lower bounds on the bending energy, which is extrinsic, i.e., immersion-dependent, in terms of the intrinsic geometric data of the surface.
An elementary and well-known lower bound on the bending energy density of an immersed surface is 
\[
|d\n|_\g^2 \ge 2|K|,
\] 
which is simply the arithmetic-geometric mean inequality for the principal curvatures.
Integrating over the surface, this implies that for every isometric immersion of $(\W,\g)$, 
\beq\label{eq:L_1_bound}
E_B(S) \ge 2\int_S |K| \VolG = 2\, \|K\|_{L^1(S)}.
\eeq
This somewhat trivial lower bound fails to detect two phenomena: concentration of curvature and energetic super-additivity.
By \Emph{concentration of curvature}, we mean that it is more energetically costly to immerse isometrically a surface with a given total curvature $\|K\|_{L^1}$, when its curvature is concentrated; the example to have in mind is a cone versus a spherical cap having the same total curvature. 
By \Emph{energetic super-additivity}, we mean that the infimal energy of an isometric immersion of a surface is generally greater than the sum of the infimal energies of isometric immersions of a partition of that surface (in the physics literature, this is sometimes known as superextensive energy).

The aim of this paper is to address the first phenomenon: obtaining a lower bound detecting curvature concentration.
Since this is a local phenomenon, we assume that the surface can be covered by a single chart $\W\subset \R^2$---that is, we consider immersions $f:\W\to \R^3$. 
We denote by $L^p(\W,\g)$ the $L^p$ norm with respect to the metric $\g$ induced by the immersion.
With this notation, for $S=f(\W)$,
\[
E_B(S) = \|d\n\|_{L^2(\W,\g)}^2.
\]
Our main result shows that, under some technical assumptions (see \thmref{thm:main_theorem}), 
\beq\label{eq:main_bound}
\|d\n\|_{L^2(\W,\g)}^2 \ge \frac{4\pi - \|K\|_{\LoneMC}}{\|K\|_{\LoneMC}} \|K\|_{\HMoneMC}^2.
\eeq
For simply-connected surfaces, the $L^1$ and $\Hminusone$ subscripts denote the standard Sobolev norms.
For multiply-connected surfaces (whence the subscript MC), they are defined such to encompass Gaussian curvature that may be ``hidden in holes''; see below for exact definitions.
As neither of the $L^1$ and the $\Hminusone$ norms control each other, this lower bound is neither stronger nor weaker than \eqref{eq:L_1_bound}.
However, it detects concentrations, and in particular, gives a sharp lower bound for the blowup of the bending energy of isometric immersions of cones (see Section~\ref{sec:examples}).
This bound extends results of Olbermann \cite{Olb17,Olb18}, whose works motivated our approach; see Section~\ref{sec:literature} for a more detailed comparison.

A main tool in proving the bound \eqref{eq:main_bound} is the estimation of the $L^1$ norm of $d\n$ along closed curves on the surface.
To this end, we derive an isoperimetric inequality pertinent to \Emph{framed loops} (\thmref{thm:GB+iso}), which we believe to be of independent interest.

There is an other instance of curvature concentration that neither of the bounds \eqref{eq:L_1_bound} and \eqref{eq:main_bound} detect: the case of \Emph{curvature dipoles}, where the curvature exhibits two loci of opposite signs.
The importance of this geometry stems from materials science, where curvature dipoles arise from a lattice defect known as an \Emph{edge-dislocation}. In this case, the presence of the defect is detected by a geometric measure known as a \Emph{Burgers vector}.

In Section~\ref{sec:examples}, we extend results of \cite{Kup17} and derive lower bounds for the bending energy based on the Burgers vector content, rather than on curvature.
For cone singularities, this lower bound detects the bending energy blowups in a way that is complementary to \eqref{eq:main_bound}---the result is weaker for small total curvatures, but stronger for large ones (Theorem~\ref{thm:cone}).
Furthermore,  and unlike \eqref{eq:main_bound}, the lower bound based on the Burgers vector detects the blowup of the bending energy for curvature dipoles (\thmref{thm:LB_disloc_optimal}). 

\subsection{Definitions}\label{sec:definitions}
In order to state our main results, we first need to introduce several definitions:

\paragraph{Immersions.}
This work concerns immersions, which may be of regularity lower than $C^1$, or be only defined in open domains:

\begin{definition}\label{def:immersions}
Let $\W\subset \R^2$ be an open domain.
We say that a differentiable map $f:\W\to \R^3$ is an \Emph{immersion} if its derivative $df$ has full rank, and the eigenvalues of $\g = f^*\euc$, viewed as a map $\W\to \Hom(\R^2;\R^2)$, are bounded and bounded away from zero. I.e., the metric and its inverse are bounded.
If $f$ is a weakly differentiable map, then we require the conditions to hold almost everywhere.
\end{definition}

Throughout this work, we denote by $C^\infty_{\Imm}(X;Y)$ the space of smooth immersions $X\to Y$, by $C^k_{\Imm}(X;Y)$ the space of $C^k$-immersions, and by $W^{m,p}_{\Imm}(X;Y)$ the space of $W^{m,p}$-immersions. 
For $p=2$, we write $H^s$ rather than $W^{s,2}$.
Of particular importance will be the space $\HtwoImm(\W;\R^3)$. 
In this case, the boundedness of the metric and its inverse imply that the unit normal is in $\Hone(\W;\Sph)$.

\paragraph{Frames along loops.}
Much of the analysis in this work concerns so-called framed loops, which are defined below. 
The analysis depends on whether a given frame along a closed curve can be obtained as a Darboux frame along the boundary of an immersed disc.
Here we explain this notion, and how it generalizes to the regularity of $\Hhalf$ frames which arise in our work.

We start with the smooth case. 
Let $\gamma\in C^\infty_{\Imm}(\Circ;\R^3)$ be an immersed closed curve, and let $\n\in C^\infty(\Circ;\Sph)$ be a map satisfying
\[
\n \perp \gamma' .
\]
We call such a pair $(\gamma,\n)$ a \textbf{framed loop}, as it defines a \Emph{moving trihedron}
\[
(\frakt,\frakn,\n) \in C^\infty(\Circ;\SO(3)),
\]
with 
\beq
\label{eq:Darboux}
\frakt = \gamma'/|\gamma'|
\Textand
\frakn = \n\times\frakt.
\eeq
Note that not every map in $C^\infty(\Circ;\SO(3))$ is induced by framed loops (for example, constant maps are not). 

\begin{definition}[Extendable framed loop]
\label{def:extandable_frame_smooth}
We say that a smooth framed loop $(\gamma,\n)$ is \Emph{extendable} if there exists an immersed disc $f\in C_{\Imm}^\infty(D; \R^3)$, such that $\gamma=f|_{\pl D}$, and $\n$ is the restriction to $\pl D\simeq\Circ$ of the normal of $f$.
\end{definition}

Thus, a framed loop is extendable if its induced moving trihedron can be realized as a Darboux frame at the boundary of an immersed disc.
Extendability has a more topological characterization (see Section~\ref{sec:Darboux} for a proof):

\begin{proposition}\label{prop:extendabe_eq_nontrivial}
A smooth framed loop $(\gamma,\n)$ is extendable 
if and only if it is isotopic to the framed loop $(\gamma_0,\n_0)$, where
\beq\label{eq:trivial_frame}
\gamma_0(\vp) = (\cos\vp,\sin\vp,0) \qquad \n_0 = (0,0,1).
\eeq
Here, isotopy means that the loops can be connected by a smooth homotopy preserving the orthogonality of $\gamma'$ and $\n$.
\end{proposition}

The isotopy of two framed loops $(\gamma_i,\n_i)_{i=1,2}$ implies the isotopy of their associated trihedra $(\frakt_i,\frakn_i,\n_i)_{i=1,2}$.
As the fundamental group of $\SO(3)$ is isomorphic to $\bbZ/2\bbZ$, the topological space $C^\infty(\Circ,\SO(3))$ has two connected components.
The trihedron $(\frakt_0,\frakn_0,\n_0)$ associated with $(\gamma_0,\n_0)$ is a generator (i.e., a non-trivial element) of $\bbZ/2\bbZ$.
It turns out that the space of framed loops also has two connected components, i.e., that $(\gamma_i,\n_i)_{i=1,2}$ are isotopic if and only if $(\frakt_i,\frakn_i,\n_i)_{i=1,2}$ belong to the same connected component of $C^\infty(\Circ,\SO(3))$ \cite[Proposition~3.2.4, Corollary~3.2.5]{Nee16}.
Thus, we conclude:

\begin{corollary}
A smooth framed loop $(\gamma,\n)$ is extendable if and only if its associated trihedron $(\frakt,\frakn,\n)$ is non-trivial.
\end{corollary}

The natural space for our work are frames for which the associated trihedra have $\Hhalf$-regularity, that is, framed loops $(\gamma,\n)$ where $\gamma\in \HthreehalvesImm(\Circ;\R^3)$ and $\n\in \Hhalf(\Circ;\Sph)$, in which the orthogonality of $\gamma'$ and $\n$ holds almost-everywhere; for brevity, we will refer to such framed loops as \Emph{$\Hhalf$-framed loops}.
In Section~\ref{sec:Darboux} we show that the definition of extendability extends naturally to this regularity.
Thus we define:

\begin{definition}[Extendable $H^{1/2}$-loop]
\label{def:extendableH32}
Let $(\gamma,\n)$ be a $\Hhalf$-framed loop.
We say that it is \Emph{extendable} if its associated moving trihedron $(\frakt,\frakn,\n) \in \Hhalf(\Circ;\SO(3))$ is \Emph{non-trivial}, i.e., is in the same connected component of $\Hhalf(\Circ;\SO(3))$ as $(\frakt_0,\frakn_0,\n_0)$.
\end{definition}

A framed loop $(\gamma,\n)$  has an associated geodesic curvature, 
\[
\kg = \ip{\frakt'/|\gamma'|,\frakn}
\]
and consequently a total geodesic curvature,
\beq\label{eq:total_kg}
\int_{\Circ} \kg\,d\ell = \int_0^{2\pi} \ip{\frakt'(\vp),\frakn(\vp)}\, d\vp =\frakt'[\frakn],
\eeq
where $d\ell$ is the length element, and the last equality is understood as the action of the distribution represented by $\frakt'$ on $\frakn$.
This last equality shows that the total geodesic curvature extends to  $(\frakt,\frakn,\n)\in \Hhalf(\Circ;\SO(3))$, since  $\frakt'\in \Hminushalf(\Circ;\R^3)$ and  $\frakn\in \Hhalf(\Circ;\Sph)$. 
Note however that it is only the integral of the geodesic curvature that can be extended, and not its point values. 

Finally, the notion of extendability extends to immersions of $2$-dimensional domains in the following sense:

\begin{definition}[Extendable immersed surface]
\label{def:W22_extendable}
Let $\W\subset \R^2$ be an open bounded domain, and let $f\in \HtwoImm(\W;\R^3)$. Denote by $\n\in \Hone(\W;\Sph)$ its unit normal.
We say that $f$ is \Emph{extendable} if for every $C^{1,1}$ simple closed curve $\Gamma\subset \bW$, the trace $(f,\n)|_{\Gamma}$ is extendable according to Definition~\ref{def:extendableH32}.
\end{definition}

The $C^{1,1}$ regularity in the above definition is the regularity needed for the trace operator to map $H^{s}(\W)$ into $H^{s-1/2}(\Gamma)$ for $s=1,2$ \cite[Theorem~1.5.1.2]{Gri85}.

It follows from \lemref{lem:extendability_homotopy} below, that it is sufficient to check the extendability of $f\in \HtwoImm(\W;R^3)$ for $C^{1,1}$ generators of the fundamental group of $\W$.
In particular, if $\W$ is simply-connected then every $f$ is extendable, and if $\W$ has a $C^{1,1}$ boundary, then it is sufficient to check extendability for each connected component of the boundary.

\paragraph{Curvature norms on finitely-connected domains.}
Let $\W\subset\R^2$ be an open Lipschitz domain homeomorphic to a disc $D$ punctured by $m$ (possibly zero) closed, disjoint discs $\bD_i$.
Such domains are called \textbf{finitely-connected}.
Denote by $\Gamma_0 = \partial D$ and $\Gamma_i = \partial D_i$, $i=1,\dots,m$ the outer and inner boundaries. 
Let $f\in C^2_{\Imm}(\bW;\R^3)$, and denote by $\g=f^*\euc$ the induced metric. 
Suppose first that $\Gamma_i$ are $C^1$-regular. For $i=1,\dots,m$, set
\beq
K_i = 2\pi - \int_{\Gamma_i} \kg\, \iVolG
\qquad
i=1,\dots,m,
\label{eq:Ki}
\eeq
where $\kg$ is the geodesic curvature of $\Gamma_i$, and $\iVolG$ is the induced length form on $\Gamma_i$.
In conjunction with the Gauss-Bonnet theorem, the quantity $K_i$ represent the total Gaussian curvature ``enclosed in the $i$-th  hole".
If $\Gamma_i$ has only Lipschitz regularity, we can define $K_i$ by taking a smooth simple curve $\tGamma_i\subset \W$ homotopic to $\Gamma_i$, 
and define
\beq\label{eq:K_i_flexible}
K_i = 2\pi - \int_{\tGamma_i} \kg\, \iVolG - \int_{\tW} K \,\VolG,
\eeq
where $\tW$ is the annulus bounded by $\Gamma_i$ and $\tGamma_i$.
Similarly, $K_i$ is well-defined for $f\in \HtwoImm(\W;\R^3)$, as in this case $\g\in H^1\cap L^\infty$, hence one can find $\tGamma_i$ as above such that $\kg\in L^2(\tGamma_i)$,  and $K\in L^1(\W)$ by Gauss' Theorema Egregium (see the proof of \lemref{lem:immersion_convergence} for similar arguments).

With this definition, we think of the curvature $K$ as a function $K:\W\times \{1,\ldots,m\} \to \R$ and define the $L^1$ and $H^{-1}$ norms as follows: 
we define the total absolute curvature enclosed in $(\W,\g)$ by
\beq
\label{eq:LoneMC}
\|K\|_{\LoneMC} = \int_\W |K|\,\VolG +  \sum_{i=1}^m |K_i|.
\eeq
To define the $\HMoneMC$-norm, we let the space of test functions be the space $\HoneMC$ of $H^1$-functions vanishing on $\Gamma_0$ and constant on $\Gamma_i$, $i=1,\dots,m$, endowed with the homogenous norm $\|\cdot\|_{\dotHnorm}$. 
We define the norm of $K$ dual to $\HoneMC$ by 
\beq
\label{eq:HMoneMC}
\|K\|_{\HMoneMC} = \sup_{\vp\in \HoneMC} \frac{\int_\W \vp K\, \VolG + \sum_{i=1}^m K_i \vp|_{\Gamma_i}}{\|\vp\|_{\dotHnorm}}.
\eeq

As with the standard $H^{-1}$ norm, this norm equals to the $\dotHnorm$ of the solution $u_K$ of the Poisson equation $-\Delta_\g u =K$ with appropriate boundary conditions (this is the ``Riesz representer'' of $K$), see \eqref{eq:ELPDE} below.
In Appendix~\ref{sect:AppB} we study this problem, prove the existence of a solution and some standard estimates.
This duality provides an electrostatics interpretation of $\|K\|_{\HMoneMC}$: Viewing $K$ has a distribution of electric charge, the function $u_K$ is the electric potential within that domain assuming that its outer boundary is grounded and that its inner boundaries are  conductors loaded with a total charge $K_i$. 
From this perspective, $\|K\|_{\HMoneMC}^2$ is the corresponding electrostatic energy.  

\subsection{Main results: curvature estimates}\label{sec:results}
Our main result is the following:

\begin{theorem}
\label{thm:main_theorem}
Let $\W\subset\R^2$ be a finitely-connected, bounded, Lipschitz domain, and let $f\in \HtwoImm(\W;\R^3)$ be an extendable immersion satisfying \underline{one} of the following additional conditions:
\begin{enumerate}
\item $f\in C^1(\W;\R^3)$. 
\item $\g = f^*\euc\in W^{2,p}(\W)$ for $p>1$.
\end{enumerate}
Then
\[
\|d\n\|_{L^2(\W,\g)}^2 \ge \frac{4\pi - \|K\|_{\LoneMC}}{\|K\|_{\LoneMC}} \|K\|_{\HMoneMC}^2.
\]
\end{theorem}

\begin{comments}
\begin{enumerate}[itemsep=0pt,label=(\alph*)]
\item For a simply-connected domain, a corollary is that $K\in H^{-1}$, which is not immediate for $f\in \HtwoImm(\W;\R^3)\cap C^1(\W;\R^3)$. 
Indeed, if $\|K\|_{L^1}<4\pi$ this is immediate, and otherwise one can partition the domain into finitely-many parts satisfying this bound.
\item For simply-connected domains, the triviality of the bound in \thmref{thm:main_theorem} for $\|K\|_{\LoneMC} \ge 4\pi$ can be relieved, 
as the bending energy can be bounded from below by the bending energy of a submanifold having total absolute curvature less than $4\pi$. 
It is however not clear how to exploit in an optimal way a partition of the surface into subdomains.
For multiply-connected domains, the prefactor $4\pi - \|K\|_{\LoneMC}$ is, in general, sharp; see comment at the end of Section~\ref{sec:cone}.
\item It is possible to extend our analysis to the non-extendable case. 
For example, if $\W$ is an annulus and $f$  is non-extendable, then the same bound holds, replacing $K_1$ with $K_1-2\pi$ in the definitions of $\|K\|_{\LoneMC}$ and $\|K\|_{\HMoneMC}$.
\item Neither of the two additional conditions implies the other:
If $f\in \HtwoImm\cap C^1$ then its metric is in general only in $\Hone\cap C^0$. 
On the other hand, a $\HtwoImm$ map may fail to be $C^1$, even if the induced metric is smooth, and a fortiori if it is only $W^{2,p}$ (except in special cases such as flat \cite{MP05} or elliptic \cite[Theorem~2.1]{HV18} metrics).
We do not know whether the theorem can be extended to $f\in \HtwoImm$ with a continuous metric, which seems to be the edge case in which the statement is still well-defined.
\item The bending energy, as well as both the $\|K\|_{\LoneMC}$ and $\|K\|_{\HMoneMC}$ norms are dimensionless. 
As a result there is not way to infer the functional dependence of the lower bound on those norms based on dimensional arguments.   
\end{enumerate}
\end{comments}

The proof of \thmref{thm:main_theorem} combines two parts: an isoperimetric inequality and an ``electrostatics'' analysis.
We derive an isoperimetric inequality pertinent to framed loops, which yields a lower bound for the derivative of $\n$ along the loop.
The ``electrostatics'' part involves the foliation of $\W$ with level sets of the ``electric potential'' $u_K$ discussed above (the Riesz representer of $K$), and the application of the isoperimetric inequality on these level sets.
The two together yield the desired lower bound for smooth enough immersions. 
The extension to $f\in \HtwoImm\cap C^1$  is obtained by a direct approximation. 
The extension to $f\in \HtwoImm$ maps with $W^{2,p}$ metrics is obtained by approximating the metric by smooth metrics, and using the level sets of a potential obtained by solving the Poisson equation using the approximate metric.

We believe that the first part of the proof, which is the following isoperimetric inequality, is of independent interest:

\begin{theorem}
\label{thm:GB+iso}
Let $(\gamma,\n)$ be an $\Hhalf$-framed loop.
Denote by $d\ell = |\gamma'(\vp)|\, d\vp$ the arclength form and by $D_s \n = \n'/|\gamma'|$ the arclength derivative of $\n$.
If $(\gamma,\n)$ is extendable, then
\beq
\brk{\int_{\Circ} |D_s \n|\, d\ell}^2 \ge 
\brk{4\pi - \Abs{2\pi - \int_{\Circ} \kg\, d\ell}} \Abs{2\pi - \int_{\Circ} \kg\, d\ell},
\label{eq:GB+iso}
\eeq
where $\int_{\Circ} \kg\, d\ell$ the total geodesic curvature of $\gamma$ as defined in \eqref{eq:total_kg}.
Otherwise, if $(\gamma,\n)$ is non-extendable, then
\beq
\brk{\int_{\Circ} |D_s\n|\, d\ell}^2 \ge 
\brk{4\pi - \Abs{\int_{\Circ} \kg\, d\ell}} \Abs{ \int_{\Circ} \kg\, d\ell}.
\label{eq:GB+iso2}
\eeq
\end{theorem}

\begin{comments}
\begin{enumerate}[itemsep=0pt,label=(\alph*)]
\item The left-hand sides of \eqref{eq:GB+iso} and \eqref{eq:GB+iso2} are finite if and only if $\n\in W^{1,1}(\Circ;\Sph)$.
\item
By Definition~\ref{def:extandable_frame_smooth}, if $(\gamma,\n)$ is extendable and smooth, then there exists an immersion $f\in C^\infty_{\Imm}(\bD;\R^3)$ such that $\gamma = f|_{\dD}$ and $\n$ is the restriction of its Gauss map at the boundary.  
Let $f$ be any such immersion, and denote by $\g$ the induced metric and by $K$ the corresponding Gaussian curvature. 
It follows from the Gauss-Bonnet theorem \cite[p.~264]{Doc76} that \eqref{eq:GB+iso} can be rewritten in the form
\beq
\brk{\int_{\dD} |d\n|_{\g}\, \iVolG}^2 \ge 
\brk{4\pi - \Abs{\int_D K\,\VolG}} \Abs{\int_D K\,\VolG}.
\label{eq:GB+isoK}
\eeq
In fact, we first prove \eqref{eq:GB+isoK}, and the inequality \eqref{eq:GB+iso} is obtained by interpreting the moving trihedron induced by $(\gamma,\n)$ as the Darboux frame of the boundary of an immersed disc.
Interestingly, the extension of the moving trihedron into an immersed disc is essential to the proof, even though the end result, \eqref{eq:GB+iso}, is oblivious to the extension.
\end{enumerate}
\end{comments}

The proof of \thmref{thm:GB+iso} is obtained by first 
identifying the left-hand side in \eqref{eq:GB+iso} as the length of the image in $\Sph$ of the Gauss map.
Weiner's isoperimetric inequality for non-simple closed curves on the sphere \cite{Wei74}  lower bounds the length of that curve by an integral of an appropriate notion of its winding number; 
we then consider an immersion $f:D\to \R^3$ whose boundary is the given framed loop (assuming it is extendable), identifying the winding number of the boundary with the degree of the Gauss map, estimating the integral of the degree using Gauss's theorem, and then using Gauss--Bonnet to obtain the final result.
For non-extendable framed loops, we first make them extendable by concatenating them with the extendable framed loop $(\gamma_0,\n_0)$ given by \eqref{eq:trivial_frame} (in the language of $\pi_1(\SO(3))$, we concatenate a trivial loop with a generator).

\subsection{Main results: Burgers circuits estimates}

\thmref{thm:GB+iso} is vacuous when the integral of the geodesic curvature is too large; namely larger than $6\pi$ in the extendable case or $4\pi$ in the non-extendable case. 
Even more importantly, it is vacuous when the integral of the geodesic curvature is exactly $2\pi$, which 
is the case when  the total signed curvature is zero, as in the presence of curvature dipoles, as appearing in many  physical systems involving dislocations (see Section~\ref{subsec:examples}).
To treat these cases, we complement \thmref{thm:GB+iso} by another lower bound pertinent to framed loops, one based on the detection of torsion, rather than the detection of curvature.
Guided by the crystallographic theory of dislocations,
this bound is based on a notion of \Emph{Burgers vector} (for the relation between Burgers vectors and torsion see \cite{KM15,KM16,EKM20,KO20}).

A smooth framed loop $(\gamma,\n)$ induces a parallel transport of vectors perpendicular to $\n$, i.e., in the span of $\frakt$ and $\frakn$. A vector field $X\in C^\infty(\Circ;\R^3)$ of the form 
\[
X = a\, \frakt + b\, \frakn
\qquad
a,b\in C^\infty(\Circ),
\]
is parallel with respect to the induced moving trihedron if 
\beq
X' = \ip{X',\n}\n =  - \ip{X,\n'}\n,
\label{eq:parallel_transport}
\eeq
where the second equality follows from the fact that $\ip{X,\n}=0$.
This defines a parallel transport operator 
\[
\Pi^{\gamma,\n}_{\vp_1,\vp_2}: \span\{\frakt(\vp_1),\frakn(\vp_1)\} \to
\span\{\frakt(\vp_2),\frakn(\vp_2)\}
\]
transporting vectors normal to $\n$ from $\vp_1\in\Circ$ to $\vp_2\in\Circ$ according to \eqref{eq:parallel_transport} along the shortest negatively-oriented path.
That is, for $v=a_0\, \frakt(\vp_1)+ b_0\, \frakn(\vp_1)$, the vector $\Pi^{\gamma,\n}_{\vp_1,\vp_2}v$ is the evaluation at $\vp_2$ of the solution to \eqref{eq:parallel_transport} with initial condition $v$.
It can be shown that the parallel transport depends, in fact, only on the geodesic curvature $\kg$, and, if $(\gamma,\n)$ induces a Darboux frame at the boundary of an immersed surface, the parallel transport defined by \eqref{eq:parallel_transport} coincides with that of the surface.

With that, we assign a notion of a Burgers vector to framed loops, as the integral of the velocity $\gamma'$ of the loop, which is parallel-transported from its point of evaluation to the base point $\vp_0$ (see Figure~\ref{fig:Burgers}):

\begin{definition}[Burgers vector]
\label{def:burgers}
The Burgers vector associated with a smooth framed loop $(\gamma,\n)$ relative to a base point $\vp_0\in\Circ$ is  a vector in $\span\{\frakt(\vp_0),\frakn(\vp_0)\}\subset\R^3$ given by 
\beq
B_{\gamma,\n,\vp_0} = \int_{\Circ} \Pi^{\gamma,\n}_{\vp,\vp_0} \gamma'(\vp)\, d\vp.
\label{eq:def_burgers}
\eeq
\end{definition}

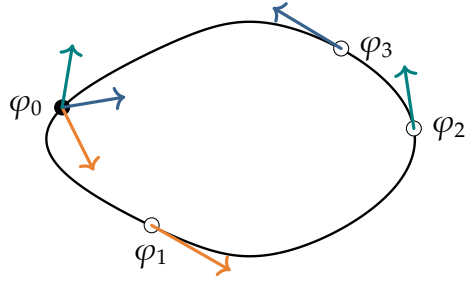
\begin{figure}[h]
\[
\begin{tikzpicture}[scale=1.4, line cap=round, line join=round]

  \def\a{1.6}   
  \def\b{1.1}   
  \def\dist{0.15}  

  \draw[line width=0.9pt]
    plot[domain=0:360, samples=240, variable=\t, smooth]
      ({\a*cos(\t) + \dist*cos(3*\t)},
       {\b*sin(\t) + 0.6*\dist*sin(2*\t)});
       
  \draw (1.05,0.86) node[circle,draw=black, fill = white,inner sep=2pt, label=right:$\vp_3$] {};
  \draw (1.74,0.1) node[circle,draw=black, fill = white,inner sep=2pt, label=right:$\vp_2$] {};
  \draw (-0.75,-0.82) node[circle,draw=black, fill = white,inner sep=2pt, label=below:$\vp_1$] {};
  \draw (-1.6,0.3) node[circle,draw=black, fill = black,inner sep=2pt, label=left:$\vp_0$] {};
	\draw[->, color = ocre, line width = 1.2pt] (-0.75,-0.82) -- (0.,-1.25);
	\draw[->, color = ocre, line width = 1.2pt] (-1.6,0.3) -- (-1.3,-0.3);
	\draw[->, color = warmblue, line width = 1.2pt] (1.05,0.86) -- (0.4,1.25);
	\draw[->, color = warmblue, line width = 1.2pt] (-1.6,0.3) -- (-1,0.4);
	\draw[->, color = teal, line width = 1.2pt] (1.74,0.1) -- (1.65,0.7);
	\draw[->, color = teal, line width = 1.2pt] (-1.6,0.3) -- (-1.50,0.9);

\end{tikzpicture}
\]
\caption{The Burgers vector of the curve is calculated by integrating its velocity parallel-transported to a reference point $\vp_0\in\Circ$.}
\label{fig:Burgers}
\end{figure}

Since $\Pi^{\gamma,\n}_{\vp_0,\vp_1}\circ \Pi^{\gamma,\n}_{\vp,\vp_0} = \Pi^{\gamma,\n}_{\vp,\vp_1}$ for every $\vp,\vp_0,\vp_1\in\Circ$, it follows that 
\[
B_{\gamma,\n,\vp_1} = \Pi^{\gamma,\n}_{\vp_0,\vp_1} B_{\gamma,\n,\vp_0},
\]
hence, since parallel transport is norm preserving, the magnitude of the Burgers vector is independent of the base point.
Moreover, it follows immediately from \eqref{eq:def_burgers} that $|B_{\gamma,\n,\vp_0}| \le \Len(\gamma(\Circ))$.

While the parallel transport operator is not well-defined pointwise for $\Hhalf$-framed loops, the notion of a Burgers vector is:

\begin{proposition}\label{prop:Burgers_Hhalf}
An equivalent definition for the Burgers vector associated with a framed loop $(\gamma,\n)$ is 
\[
B_{\gamma,\n,0} = \int_{\Circ}  \ip{\calL(\vp),\n'(\vp)} \n(\vp) \, d\vp  = \n'[\calL\otimes \n],
\]
where 
\[
\calL(\vp) = \int_{\vp}^{2\pi} \Pi^{\gamma,\n}_{\vp',\vp} \gamma'(\vp')\, d\vp',
\]
and $\n'$ is viewed as a distribution acting on the tensor-valued function $\calL\otimes \n$.
This definition extends to $\Hhalf$-framed loops, as $\n\in H^{1/2}$ and $\calL\in H^{3/2}$.
\end{proposition}

The proof is given in Section~\ref{sec:Darboux}.
The following theorem is an analog of \thmref{thm:GB+iso}, with the Burgers vector rather than an enclosed curvature as a source of non-flatness: 

\begin{theorem}
\label{thm:burgers}
Let $(\gamma,\n)$ be an $\Hhalf$-framed loop. 
Then,
\beq
\int_{\Circ} |D_s \n|\, d\ell \ge \frac{|B_{\gamma,\n,\vp_0}|}{\Len(\gamma(\Circ))},
\label{eq:burgers_bound}
\eeq
where $D_s\n$ and $d\ell$ are defined as in \thmref{thm:GB+iso}.
\end{theorem}

The proof is given in \secref{sec:burgers}. Note the absence of any topological constraint in comparison with \thmref{thm:GB+iso}.
Given a immersion $f\in \HtwoImm(\W;\R^3)\cap W^{1,\infty}(\W;\R^3)$, one can foliate $\W$ with simple, closed curves, and apply \eqref{eq:burgers_bound} to each curve in conjunction with the coarea formula, obtaining a lower bound for $\|d\n\|_{L^2(\W,\g)}^2$. 
Unlike in the proof of \thmref{thm:main_theorem} where a specific foliation is constructed in order to obtain a negative Sobolev norm, we do not have a prescribed procedure for foliating the surface for the Burgers vector based bound.

\subsection{Examples}\label{subsec:examples}

\paragraph{Cones.} 
Consider a cone with angular deficit/excess $\alpha \in (-\infty,2\pi)$, inner radius $r_0 > 0$ and outer radius $R > 0$ (Figure~\ref{fig:cone} depicting the case of $\alpha>0$).
See Section~\ref{sec:cone} for an explicit definition. 
$\alpha$ is also the total curvature enclosed in the hole as defined by \eqref{eq:Ki}. 
The limit $r_0\to 0$ can be viewed as the prototypical case of curvature concentration, whose energy blowup the trivial bound \eqref{eq:L_1_bound} does not detect.
The following theorem exemplifies our main results for this geometry: 

\newcommand{\drawsector}[6]{%
  \begin{scope}
    \fill[#6] (#1,#2) -- 
      ({#1 + #3*cos(#4)},{#2 + #3*sin(#4)}) arc (#4:#5:#3) -- cycle;
  \end{scope}
}

\begin{figure}[h]
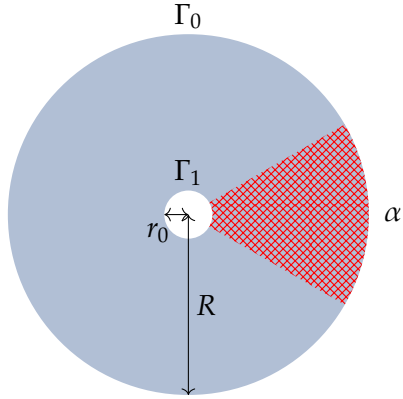

\begin{center}
\btkz[scale=0.8]
\fill[color=warmblue!30] (0,0) circle (3cm);
\drawsector{0}{0}{3}{-30}{30}{pattern=crosshatch, pattern color=red};
\fill[color=white] (0,0) circle (0.4cm);
\draw[<->] (0, 0) -- (0,-3); 
\draw[<->] (0, 0) -- (-0.4,0); 
\node at (0.3,-1.5){{\small $R$}};
\node at (-0.5,-0.3){{\small $r_0$}};
\node at (3.4,0){{\small $\alpha$}};
\node at (0.,0.7){{\small $\Gamma_1$}};
\node at (0,3.3){{\small $\Gamma_0$}};
\etkz 
\end{center}
\caption{A cone geometry with inner radius $r_0$, outer radius $R$, and a deficit angle $\alpha>0$; the edges of the removed sector are glued to each other.}
\label{fig:cone}
\end{figure}

\begin{theorem}\label{thm:cone}
Let $(\W,\g)$ be a cone with total curvature $\alpha\in (-\infty,2\pi)$, inner radius $r_0$ and outer radius $R$.
Let $f\in \HtwoImm(\W;\R^3)$ be an isometric immersion of $(\W,\g)$.
Then:
\begin{enumerate}
\item Assuming $f$ is extendable, \thmref{thm:main_theorem} yields
	\beq
	\|d\n\|_{L^2(\M)}^2 \ge \brk{4\pi -  |\alpha|} 
	 \frac{|\alpha|}{2\pi - \alpha} \log\frac{R}{r_0}.
	\label{eq:cone_lb_curvature}
	\eeq
\item Applying \thmref{thm:burgers} on a foliation of $\W$ yields 
	\beq
	\|d\n\|_{L^2(\M)}^2 \ge \frac{4\sin^2(\alpha/2)}{(2\pi- \alpha)^3} \, \log \frac{R}{r_0}.
	\label{eq:cone_lb_Burgers}
	\eeq
\end{enumerate}
\end{theorem}  
The proof is given in Section~\ref{sec:cone}.
Note that both bounds diverge like $\log(R/r_0)$ as $R/r_0\to \infty$.
For small values of $|\alpha|$, the bound \eqref{eq:cone_lb_curvature} is superior to \eqref{eq:cone_lb_Burgers}, as it grows like $|\alpha|$, rather then $|\alpha|^2$. This linear dependence on $|\alpha|$ is also tight.
However, \eqref{eq:cone_lb_curvature} breaks down for $\alpha\le -4\pi$, whereas \eqref{eq:cone_lb_Burgers} does not, unless $\alpha=-2\pi k$ for some $k\in \bbN$.
Indeed, for $\alpha=-2\pi k$ there exists smooth immersions with zero bending.
The $\log(R/r_0)$ divergence and the $|\alpha|$ growth for small curvatures are both tight, as explicit constructions show (in the case of $\alpha>0$, the ``standard'' embedded cone); this will be detailed in an upcoming work.
We discuss this, as well as consequences for the tightness of the pre-factors in Theorems~\ref{thm:main_theorem}--\ref{thm:GB+iso}, in Section~\ref{sec:cone}. 

\paragraph{Dislocations: cone dipoles.}
We proceed to consider isometric immersions of a surface with a conical dipole, which is a continuum analog of a lattice exhibiting an edge-dislocation. 
The geometry is depicted in \figref{fig:dislocation}:
Take a Euclidean disc of radius $R$, with at its center two circular holes of radius $r_0$ at a distance $r_0$ apart. With one of the holes as its center, remove a sector such to form a cone with angle deficit $\alpha$. Then, cut a straight segment from the outer boundary up to the boundary of the other hole, and insert that same sector such to form an E-cone with excess angle $\alpha$.

\begin{figure}[h]
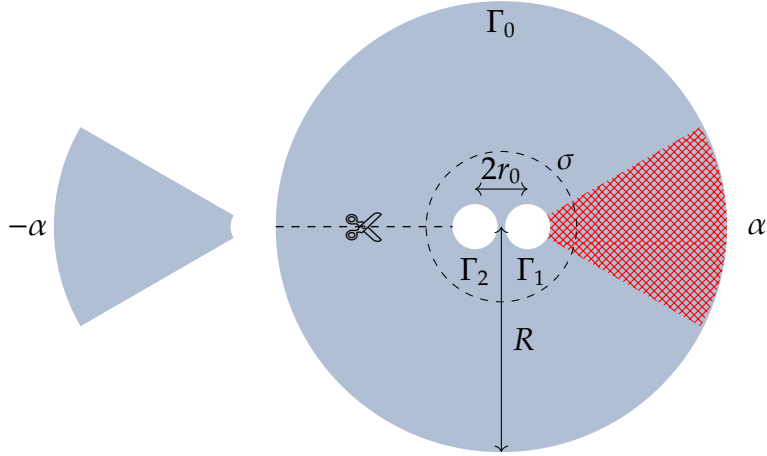

\[
\btkz
\fill[color=warmblue!30] (0,0) circle (3cm);
\drawsector{0.35}{0}{2.65}{-30}{30}{pattern=crosshatch, pattern color=red};
\fill[color=white] (-0.35,0) circle (0.3cm);
\fill[color=white] (+0.35,0) circle (0.3cm);
\drawsector{-3.3}{0}{2.65}{150}{210}{color = warmblue!30};
\fill[color=white] (-3.3,0) circle (0.3cm);
\draw[dashed, line width=0.5pt, postaction={
    decorate,
    decoration={
        markings,
        mark=at position 0.5 with {\node[rotate=0, inner sep=0pt] {\ScissorHollowRight};}
    }
}] (-3,0) -- (-0.65,0);
\draw[<->] (-0.35, 0.5) -- (0.35,0.5); 
\node at (0,0.75){$2r_0$};
\draw[<->] (0, 0) -- (0,-3); 
\node at (0.3,-1.5){$R$};
\node at (3.4,0){$\alpha$};
\node at (-6.3,0){$-\alpha$};
\node at (0.4,-0.6){$\Gamma_1$};
\node at (-0.35,-0.6){$\Gamma_2$};
\node at (0,2.7){$\Gamma_0$};
\draw[dashed] (0,0) circle (1cm);
\node at (0.86,0.86){$\sigma$};

\etkz 
\]
\caption{The geometry of an edge-dislocation as described in the text.}
\label{fig:dislocation}
\end{figure}
 
The result is a surface $\W$ endowed with a locally-Euclidean metric $\g$, enclosing two sources of curvature of opposite magnitudes.
This surface cannot be isometrically immersed in the plane, even if one cuts out a region containing the two holes (e.g., cutting out the domain enclosed by the curve $\sigma$ in Figure~\ref{fig:dislocation}) \cite{KMS15,KM25}.
In this geometry, the magnitude of the Burgers vector of any simple curve enclosing both holes is given by \cite{KM15}
\[
\e = 2r_0 \sin(\alpha/2).
\]

A longstanding question in the physics literature has been the behavior of the infimal bending energy as $R\to\infty$ with $\alpha$ and $r_0$ fixed.  Work by Nelson and Peliti \cite{NP87}, followed by a numerical study by Seung and Nelson \cite{SN88},  led to a conjecture whereby the infimal bending energy remains bounded as $R\to\infty$ (the numerical study allowed also for stretching deformations, but led nevertheless to the stated conjecture). This conjecture was proved false in \cite{Kup17}, where a $\log (R/r_0)$ lower bound was proved.

\thmref{thm:main_theorem} is not strong enough to detect this logarithm divergence: 
Indeed, applying it to this geometry yields the estimate
\beq
\|d\n\|_{L^2(\W,\g)}^2 \ge \frac{4\pi - 2\alpha}{2\alpha} \|K\|_{\HMoneMC}^2.
\label{eq:LB_disc_curvature}
\eeq
The norm $\|K\|_{\HMoneMC}$ coincides with the electrostatic energy in a domain of diameter $O(R)$, having its outer boundary grounded, and a charge dipole at its center.
As is well-known in the electrostatic context, the electric potential induced by a charge dipole decays like $1/r^2$, which implies that $\|K\|_{\HMoneMC}$ does not diverge as $R\to \infty$.
Even worse, if one cuts out a region containing the two holes, the total curvature inside the single hole is zero, hence in this case \thmref{thm:main_theorem} fails to detect \emph{any} bending energy.

Nevertheless, in Section~\ref{sec:dislocations} we show that \thmref{thm:burgers} detects the bending energy in this case: 
we show that a naive application of \thmref{thm:burgers} on a foliation of curves yields the bound 
\beq\label{eq:dislocation_naive_lb}
\|d\n\|_{L^2(\W,\g)}^2 \ge C\frac{ \e^2}{r_0^2},
\eeq
which does not diverge with $R/r_0\to \infty$.
However, a more delicate analysis, based on \thmref{thm:burgers}, yet taking into account the energetic super-additivity in this geometry, yields an optimal bound, and retrieves the divergence established in \cite{Kup17} in a considerably more transparent way, which is also better aligned with the rest of this work:

\begin{theorem}
\label{thm:LB_disloc_optimal}
There exists $C>0$ such that for the metric of the curvature dipole, as depicted in Figure~\ref{fig:dislocation}, the bending energy satisfies
\[
\|d\n\|_{L^2(\W,\g)}^2 \ge \frac{C\e^2}{r_0^2} \log\frac{R}{r_0}.
\]
This bound holds even if one cuts out a region or size $O(r_0)$ containing the two holes, and is optimal in its dependence on $R/r_0$.\end{theorem}

An upper bound with the same scaling can be obtained by connecting two half-disclinations of opposite change by flat ridges \cite{GHKM13}. 
The prefactor $\e/r_0$ is always bounded by a universal constant \cite[Section~3.1]{KM25}, and in most application  is also bounded from below. 

\subsection{Related work}\label{sec:literature}

\paragraph{Comparison with results of Olbermann on cones.}
Inequality \eqref{eq:GB+isoK}  should be compared with the isoperimetric inequality obtained by Olbermann \cite[Lemma~1]{Olb17}:

\begin{quote}
Let $D\subset\R^2$ be the unit disc and let $u\in C^2(\bD)$. Then,
\beq
\brk{\int_{\dD} |D^2 u|\, ds}^2 \ge 4\pi \Abs{\int_D \det D^2 u\, dx}.
\label{eq:olbermann}
\eeq
\end{quote}

Inequality \eqref{eq:olbermann} can be interpreted as a linearization of \eqref{eq:GB+isoK}.
Specifically, consider an immersed surface $f\in C^\infty_{\Imm}(D;\R^3)$ obtained by the graph of a function
\[
f(x,y) = (x,y,\e^{1/2}\, u(x,y)),
\]
where we fix $u\in C^\infty(D)$ and view $\e>0$ as a small parameter.
Denote by $\g_\e$ the $\e$-dependent induced metric and by $K_\e$ the corresponding Gaussian curvature. 
By \eqref{eq:GB+isoK}, for every $\e\ge0$,
\beq
\brk{\int_{\dD} |d\n_\e|_{\g_\e}\, \iVolGeps}^2 \ge 
\brk{4\pi - \Abs{\int_D K_\e\,\VolGeps}} \Abs{\int_D K_\e\,\VolGeps}.
\label{eq:GB+isoKeps}
\eeq
A direct power expansion yields,
\[
K_\e\,\VolGeps = \e \, \det D^2 u\, dx + O(\e^2),
\]
and
\[
 |d\n_\e|_{\g_\e}\,\iVolGeps = \e^{1/2} \, |D^2 u|\, ds + O(\e^{3/2}),
\]
hence \eqref{eq:olbermann} is obtained by differentiating \eqref{eq:GB+isoKeps} with respect to $\e$ at zero. 
 
The tools in the proof of \thmref{thm:GB+iso} and  \cite[Lemma~1]{Olb17} are similar, combining degree theory and an isoperimetric inequality. 
In our nonlinear setting, the analysis is carried out on the unit sphere, which is the codomain of the Gauss map. It requires the use of an isoperimetric inequality for non-simple closed curves on the sphere \cite{Wei74}, identifying the integral of $|d\n|_\g$ with the length of $\n(\dD)\subset\Sph$, and combining the appropriate notions of degree with Gauss's theorema egregium.

Our method of proof of \thmref{thm:main_theorem} hinges on some of the ideas used by Olbermann \cite{Olb17,Olb18}, with the following major differences: 
(i) Olbermann's work focuses on one specific geometry---a cone of positive curvature---whereas our analysis applies to arbitrary surfaces. 
(ii) The foliation of the surface in \cite{Olb17,Olb18}  uses the level sets of a logarithmic function; 
as our work shows, this gives an optimal bound because this function is the ''Riesz representer" of the $\Hminusone$-norm of the curvature of the truncated cone.
This observation enables the identification of the bound with a negative Sobolev norm of the curvature, and consequently the generalization to general geometries, and multiply-connected domains. 
(iii) The bending energy functional used in \cite{Olb17,Olb18}, which is the $L^2$-norm squared of the Laplacian of the configuration, is only a linearization of the true (geometric, nonlinear) bending energy.
The bound obtained by Olbermann for this (linearized) energy is similar to \eqref{eq:cone_lb_curvature} --- namely, has a logarithm divergence, and a linear dependence on $\alpha$ for small $\alpha$.

On the other hand, \cite{Olb17,Olb18} address a wider scope, in which 
one distinguishes between a ``reference" metric $\g$ and the ``actual" metric $f^*\euc$ induced by the immersion.  The reference metric is the one with respect to which norms and the volume form are evaluated. 
The total energy consists of the bending energy supplemented by a penalization---a \Emph{stretching energy}---for the discrepancy between the reference and  actual metrics. 
Extending our nonlinear analysis to allow for metric deformations is left for future work.

\paragraph{Relations to non-Euclidean elasticity.}
In a broader context, this work relates to a fundamental problem in Riemannian geometry: estimating the minimal amount of deformation required to immerse one Riemannian manifold into another. 
This problem has a direct application to \Emph{incompatible elasticity} (also known as \Emph{non-Euclidean elasticity}), where the elastic body is a $d$-dimensional Riemannian manifold with boundary $(M,\G)$, and the ambient space is a $d$-dimensional Riemannian manifold without boundary---Euclidean space $(\R^d,\euc)$ in most applications. 
A configuration of the body is an immersion $f:M\to \R^d$, and the elastic energy associated with a configuration $f$ is a measure of metric discrepancy between the actual metric of the body, $f^*\euc$, and its intrinsic reference metric $\G$ (note that unlike above, here the codimension is zero). 
Heuristically, for orientation preserving maps, this energy is of the type
\[
E_{(M,\G)} (f) = \dashint_M |f^*\euc - \G|^2 \, \VolGen_{\G},
\]
where $\dashint$ is the integral divided by the total volume.
The fundamental problem is to estimate the infimal elastic energy in terms of the geometry of the body, notably its curvature.

This problem as stated is widely open. 
A first attempt to address it is due to Kupferman and Shamai \cite{KS12}, resulting in a rather obscure lower bound applicable only in the case of positive scalar curvature. 
Lewicka and Mahadevan conjectured in \cite{LM22} that the infimal energy scales like the $H^{-2}$ norm of the Riemann curvature tensor, whereas Kupferman and Maor \cite{KM25b} showed that this conjecture holds asymptotically in the limit of vanishing curvature.

Another asymptotic regime, of important physical relevance, is that of thin bodies, where $M=M_t$ is a $t$-tubular neighborhood of a $k$-dimensional submanifold $(S,\g)$ (in applications, we typically have $d=3$ and $k=1,2$).
In this case, it was shown in \cite{MS19} that, for $t\ll 1$,
\[
\inf E_{(M_t,\g)} \gtrsim t^4 \|R^\G\|_{L^2(S)}^2,
\]
where $R^\G$ is the Riemann curvature tensor of $\G$.
This result is optimal in the case of $k=0,1$ (for $k=0$ it was generalized to general ambient space in \cite{KM25}), but not for higher $k$; indeed, it follows from dimension-reduction results \cite{LP11,KS14} that, in the case that $S$ is totally-geodesic in $M$ (so-called \Emph{non-Euclidean plate theory}), we have that
\[
\inf E_{(M_t,\g)} = t^2 \inf E_B + o(t^2),
\]
where $E_B$ is as in \eqref{eq:bending_energy} (or a higher dimensional analogue), restricted to isometric immersions of $(S,\g)$.
In this context, our work provides a new lower bound to the elastic energy of thin plates.

\paragraph{Relations to compensated compactness results.}
As $K$ is, by Gauss' theorema egregium, the Jacobian determinant of the normal $\n$, our result shows that the $\dot{H}^1$ norm of $\n$ controls an $\Hminusone$ norm of its Jacobian determinant.
This is reminiscent of the fundamental result of Coifman et al.~\cite{CLMS93} for functions in $W^{1,n}(\R^n;\R^n)$, in which the $n$-th power of the $\dot{W}^{1,n}$ norm controls the Hardy $\mathcal{H}^1$ norm of the Jacobian determinant;  note that $\mathcal{H}^1$ controls both $L^1$ and $\Hminusone$ for $n=2$ (interestingly, a precursor of this result \cite{Mul90} also uses an isoperimetric inequality to obtain higher regularity of the determinant). 
For sphere-valued maps, M\"uller and \v{S}ver\'ak \cite{MS95} has even more similar results, relating the $\dot{H}^1$ norm of a sphere-valued map to the $\mathcal{H}^1$-norm (Corollary 3.3.2) and the $\Hminusone$-norm (Proposition~3.5.5, Theorem~3.5.6) of its Jacobian determinant;
interestingly,  a constant similar to the $4\pi - \|K\|_{\LoneMC}$ prefactor in \eqref{eq:main_bound} appears there as well.
Despite these striking similarities, we were not able to identify a direct relation between these results and ours.
Note, in particular, that these estimates in \cite{MS95} do not detect the curvature ``in the holes'' which is essential to our analysis, and that their $\dot{H}^1$ and $\Hminusone$ norms are somewhat different from ours, as they consider maps from the whole plane, with its Euclidean volume form, to the sphere, whereas in our case, both the volume form on the domain and the Gauss map to the sphere depend on the underlying immersion.

\paragraph{Structure of paper}
In \secref{sec:Darboux} we prove the auxiliary results concerning the extendability of framed loops, and in particular, at the degree of regularity needed for this work.
In \secref{sec:GB+iso} we prove \thmref{thm:GB+iso}. 
In \secref{sec:main_theorem} we prove \thmref{thm:main_theorem}, first in the smooth case, and then in the two different regimes of lower regularity. 
In \secref{sec:burgers} we prove \thmref{thm:burgers}.
In \secref{sec:examples}, we consider two examples---cones and curvature dipoles---and prove the lower energy bounds Theorem~\ref{thm:cone} and Theorem~\ref{thm:LB_disloc_optimal}. Finally, we prove in the appendix technical results concerning the closure of $\Hs$ functions under products and composition, as well as elliptic estimates pertinent to finitely-connected domains; these are variations of standard results, but we were not able to locate precise references.

\paragraph{Acknowledgments}
We are grateful to Dror Bar Natan, Yoel Groman, Or Hershkovits, Tom Needham, Tal Novik, Jake Solomon and Lior Yanovski for their useful advice. RK was funded by ISF Grant 560/22,  CM was funded by ISF grant 2304/24 and BSF grant 2022076, and DPG was funded by the Zuckerman STEM Leadership Program. 
This work was written while CM was visiting the University of Toronto and the Fields Institute; CM is grateful for their hospitality.

\section{$\Hhalf$-framed loops: extendability and Burgers vectors}
\label{sec:Darboux}

We start by proving \propref{prop:extendabe_eq_nontrivial}, i.e., that extendability of a framed loop is equivalent to non-triviality of the corresponding moving trihedron in the smooth case:

\Emph{Proof of \propref{prop:extendabe_eq_nontrivial}}:
The immersion $f_0\in C^\infty_{\Imm}(\bD;\R^3)$ given in polar coordinates by
\[
f_0(r,\vp) = r(\cos\vp,\sin\vp,0) 
\] 
restricts on the boundary to $(\gamma_0,\n_0)$. 
In the first direction, let $(\gamma_1,\n_1)\sim(\gamma_0,\n_0)$. By \cite[comment after Theorem~1.1]{Hir59}, to every smooth isotopy $t\mapsto (\gamma_t,\n_t)$ connecting $(\gamma_0,\n_0)$ and $(\gamma_1,\n_1)$ there exists a corresponding variation of smooth immersions $t\mapsto f_t$ assuming these boundary data (immersions $D\to\R^3$ are said to be flexible). 
That is, the framed loop $(\gamma_1,\n_1)$ is extendable.

The reverse claim hinges on  $C^\infty_{\Imm}(\bD;\R^3)$ being connected: 
indeed, suppose that $(\gamma,\n)$ is extendable and let $f\in C^\infty_{\Imm}(\bD;\R^3)$ be a corresponding extension. 
First, using affine transformations, modify $f$ smoothly such that $f(0) = 0$ and $df_0=(\id,0)$. 
Then, consider the homotopy of immersions,
\[
H_t(r,\vp) = \frac{f(tr,\vp)}{t}
\]
starting at $t=1$. 
As $t\to 0$,  $H_t$ tends to $f_0$, hence every $f$ is isotopic to $f_0$, and consequently,  the boundary data are isotopic as well.
\hfill\ding{110}

We now turn our attention to framed loops of low-regularity:

\begin{lemma}
\label{lem:H12frames_well_defined}
Let $(\gamma,\n)$ be an $\Hhalf$-framed loop.
Then the moving trihedron $(\frakt,\frakn,\n)$ defined by \eqref{eq:Darboux} is in $\Hhalf(\Circ;\SO(3))\cap L^\infty(\Circ;\SO(3))$. 
Moreover, the map $(\gamma,\n)\mapsto (\frakt,\frakn,\n)$ is continuous as a map $\Hthreehalves_{\Imm,c,C}(\Circ;\R^3)\times \Hhalf(\Circ;\Sph) \to \Hhalf(\Circ;\SO(3))$, where $\Hthreehalves_{\Imm,c,C}(\Circ;\R^3)$ are immersions $\gamma\in \HthreehalvesImm(\Circ;\R^3)$ such that $|\gamma'|\in (c,C)$ for some fixed $C>c>0$.
\end{lemma}

\begin{proof}
Fix $C>c>0$, and denote by $B_r\in \R^3$ the ball or radius $r$ around the origin.
Note that the map $F:B_C\setminus B_c\to \Sph$, $F(v) = \frac{v}{|v|}$ is a Lipschitz map, whose Lipschitz constant depends only on $c$.
Thus, \lemref{lem:Hs_composition1}, concerning the composition of $\Hs$-maps, shows the continuity of the map 
\[
F\circ\deriv{}{\vp}: \Hthreehalves_{\Imm,c,C}(\Circ;\R^3) \to \Hhalf(\Circ;\Sph), \qquad \gamma \mapsto \frakt.
\] 
\lemref{lem:Hs_multiplication}, concerning the pointwise multiplication of $\Hs$-maps, proves the continuity of the map $(\frakt,\n)\mapsto \frakn$, which completes the proof.
\end{proof}

\begin{lemma}\label{lem:H32arclength}
Let $(\gamma,\n)$ be an $\Hhalf$-framed loop.
Then the arclength function $s:\Circ\to\Circ$
\[
s(t) = 2\pi \frac{\int_0^t |\gamma'(\theta)|\,d\theta}{\int_0^{2\pi} |\gamma'(\theta)|\,d\theta}
\] 
is an $\Hthreehalves$-diffeomorphism (i.e., a bijection having an $H^{3/2}$ inverse). 
Moreover, $(\gamma,\n)\circ s^{-1}$ is an $\Hhalf$-framed loop.
\end{lemma}

\begin{proof}
Since $\gamma\in \HthreehalvesImm(\Circ;\R^3)$, it follows that $s$ is a bi-Lipschitz map, and therefore also $s^{-1}$.
Since $(s^{-1})'(t) = 1/s'(s^{-1}(t))$, it follows from \lemref{lem:Hs_composition1} (applied to the inversion map, noting that it is Lipschitz away from the origin) and \lemref{lem:Hs_composition2} that $(s^{-1})'\in \Hhalf$, thus $s^{-1}\in \Hthreehalves$.
Similarly, it follows that $\gamma\circ s^{-1} \in \HthreehalvesImm(\Circ;\R^3)$ and $\n\circ s^{-1}\in \Hhalf(\Circ;\Sph)$, with $(\gamma\circ s^{-1})' \perp \n\circ s^{-1}$.
\end{proof}

With these two lemmas, we can now justify Definition~\ref{def:extendableH32}:

\begin{proposition}
\label{prop:connected_framed_loops}
The space of $\Hhalf$-framed loops $(\gamma,\n)$ has two connected components, corresponding to the two connected components of $\Hhalf(S^1;\SO(3))$.
Namely, any two framed loops $(\gamma_i,\n_i)_{i=1,2}$ such that their associated frames $(\frakt_i,\frakn_i,\n_i)_{i=1,2}$ are in the same connected component of $\Hhalf(S^1;\SO(3))$ are isotopic.
Moreover, the isotopy can be chosen such that the framed loops are smooth except possibly at the endpoints.
In particular, smooth framed loops are dense in the space of $\Hhalf$-framed loops, and the smooth approximating sequence can be chosen with uniform bounds from above and below on the speed.
\end{proposition}

\begin{proof}
The fact that $\Hhalf(S^1;\SO(3))$ has two connected components follows from \cite[Prop.~2.6]{vSc19}, which also implies that if $(\frakt_i,\frakn_i,\n_i)_{i=1,2}$ are in the same connected component then they can be connected by a path $(\frakt_\tau,\frakn_\tau,\n_\tau)_{\tau\in[1,2]}$ such that $(\frakt_\tau,\frakn_\tau,\n_\tau)$ is smooth for every $\tau\in (1,2)$.

The argument of \cite[Proposition~3.2.4]{Nee16} (which is based on the Smale-Hirsch theorem for the inclusion of framed loops into the space of ``formal", i.e., non-holonomic, framed loops) shows that if $(\gamma_i,\n_i)_{i=1,2}$ are framed loops of the same constant speed, then any isotopy $(\frakt_\tau,\frakn_\tau,\n_\tau)_{\tau\in[1,2]}$ of their corresponding trihedra can be adjusted such that $(\frakt_\tau,\frakn_\tau,\n_\tau)_{\tau\in[1,2]}$ are induced by constant speed framed loops $(\gamma_\tau,\n_\tau)_{\tau\in [1,2]}$, forming an isotopy between the endpoints.
This argument, while stated for smooth framed loops, works also in the regularity of the current proposition, i.e., for  $\gamma_i \in \HthreehalvesImm(\Circ;\R^3)$ and  $\n_i\in \Hhalf(\Circ;\Sph)$, forming a continuous isotopy in this space, with smooth framed loops for $\tau\in (1,2)$.
This can be easily adjusted for the case of endpoints having constant speed but not the same length (say, $(\ell_i)_{i=1,2}$), by first constructing an isotopy between $(\gamma_1,\n_1)$ and $((\ell_1/\ell_2)\gamma_2,\n_2)$, and then multiplying the resulting loops by a smooth function that interpolates between $1$ and $\ell_2/\ell_1$.

It remains to remove the "constant speed" assumption.
Let $s_i$ be the arc-length parametrizations of $\gamma_i$.
Thus, $(\tgamma_i,\tilde{\n}_i)_{i=1,2} := (\gamma_i\circ s_i^{-1},\n_i\circ s_i^{-1})_{i=1,2}$ are a pair of framed loops with constant speed, which by \lemref{lem:H32arclength} are of the same regularity as $(\gamma_i,\n_i)_{i=1,2}$.
Thus, there exists an isotopy  $(\tgamma_\tau,\tilde{\n}_\tau)_{\tau\in [1,2]}$ connecting them, where for $\tau\in (1,2)$ the framed loops are smooth.

Now, let $(s_{i,\e})_{i=1,2,\e\in (0,\e_0)}$ be an $\e$-mollification of $s_i$. 
Since mollification only improves the bi-Lipschitz constant (as it increases the minimum and decreases the maximum of $s_i'$), $(s_{i,\e})_{i=1,2,\e\in (0,\e_0)}$ are uniformly bi-Lipschitz.
Since the space of smooth, orientation preserving diffeomorphisms of $\Circ$ is connected (given such diffeomorphism $s$, consider the isotopy $\vp_\tau(t) = \tau t + (1-\tau) s(t)$ to the identity map), we can obtain a smooth path of uniformly bi-Lipschitz diffeomorphisms $(s_\tau)_{\tau\in (1,2)}$, which converge in $\Hthreehalves$ to $s_i$ when $\tau\to i$.
It follows from \lemref{lem:Hs_composition2} that the composition $(\tgamma_\tau\circ s_\tau,\tilde{\n}_\tau\circ s_\tau)_{\tau\in [1,2]}$ is a continuous isotopy connecting $(\gamma_i,\n_i)_{i=1,2}$, which completes the proof.
\end{proof}

\begin{remark}
It might be that the extendability can be defined equivalently not through the connected components of the trihedron $(\frakt,\frakn,\n)$, but directly as in Definition~\ref{def:extandable_frame_smooth}, i.e., that a framed loop $(\gamma,\n)$ with $\gamma \in \HthreehalvesImm(\Circ;\R^3)$ and  $\n\in \Hhalf(\Circ;\Sph)$ is extendable if and only if there exists $f\in \HtwoImm(D;\R^3)$ such that the restriction of the immersion and its normal to the boundary equals $(\gamma,\n)$. As this possible characterization was not needed for our purposes, we did not pursue this direction.
\end{remark}

We will also need the following refinement of the density claim, in which $\n$ has higher regularity:

\begin{lemma}
\label{lem:H32_W11_frames}
Let $(\gamma,\n)\in \HthreehalvesImm(\Circ;\R^3)\times W^{1,1}(\Circ;\Sph)$ be a framed loop. Then, there exists a sequence $(\gamma_n,\n_n)$ of smooth framed loop converging to $(\gamma,\n)$ in that topology.
\end{lemma}

\begin{proof}
By \lemref{lem:H32arclength}, we may assume that $|\gamma'|$ has constant speed, noting that the reparametrized $\n$ retains the $W^{1,1}$-regularity, as a composition of a $W^{1,1}$ map and an $\Hthreehalves$ map on a one-dimensional manifold. 

Let $\rho_\e\in C^\infty(\Circ)$ be a family of mollifiers. For $n\in\bbN$ set
\[
\gamma_n = \rho_{1/n}* \gamma
\Textand
\n_n = \frac{\rho_{1/n}* \n}{|\rho_{1/n}* \n|}.
\]
Since $\rho_{1/n}* \n \to \n\in W^{1,1}(\Circ;\Sph)$, it converges also uniformly (by the Poincare-Sobolev inequality), hence
\[
|\rho_{1/n}* \n| \to 1
\qquad
\text{uniformly}.
\] 
Likewise, $\gamma_n\to\gamma$ in $\Hthreehalves$; 
by the property of mollifiers,
\[
\gamma_n' = \rho_{1/n}*\gamma' \to \gamma' 
\qquad 
\text{in $\Hhalf(\Circ;\R^3)$}.
\]
We next show that $\gamma_n$ is an immersion for $n$ large enough. 
Fix $\vp\in\Circ$. Then,
\[
\begin{aligned}
\inf_{s\in\Circ} |\gamma_n'(\vp) - \gamma'(s)| & = \inf_{s\in\Circ} \Abs{\int_{\Circ} \rho_{1/n}(\vp - z) (\gamma'(z) - \gamma'(s))\, dz} \\
&\le \inf_{s\in\Circ} \int_{\Circ} \rho_{1/n}(\vp - z) |\gamma'(z) - \gamma'(s)|\, dz \\
&\le C \dashint_{B_{1/n}(\vp)} \dashint_{B_{1/n}(\vp)}  |\gamma'(z) - \gamma'(s)|\, dz ds,
\end{aligned}
\]
for some fixed constant $C$. The right-hand side tends to zero as $n\to\infty$ by the finiteness of the Gagliardo seminorm of $\gamma'$ and Lebesgue dominated convergence (functions in $\Hhalf(\Circ)$ are in $\VMO(\Circ)$; see \cite[Lemma~2.4]{vSc19}). Thus, $\gamma_n'$ converges uniformly to the set
\[
\gamma'(\Circ) \subset  \frac{\Len(\gamma)}{2\pi} \Sph,
\]
and therefore $\gamma_n$ is an immersion for $n$ large enough (note that $\gamma_n'$ converges uniformly to a set---not to $\gamma'$). 

Consider next the inner-product
$\ip{\gamma_n',\n_n}$
which by \lemref{lem:Hs_multiplication} satisfies
\[
\ip{\gamma_n',\n_n} \to \ip{\gamma',\n} = 0 \qquad \text{in $\Hhalf(\Circ)$}.
\] 
For every $\vp\in\Circ$,
\[
\begin{aligned}
\ip{\gamma_n',\n_n}(\vp) 
&= \frac{1}{|\rho_{1/n}*\n|(\vp)} \int_{\Circ}\int_{\Circ}
\rho_{1/n}(t)  \rho_{1/n}(s) \ip{\gamma'(\vp - t),\n(\vp-s)}\, ds dt \\ 
&= \frac{1}{|\rho_{1/n}*\n|(\vp)} \int_{\Circ}\int_{\Circ}
\rho_{1/n}(t)  \rho_{1/n}(s) \ip{\gamma'(\vp - t),\n(\vp-s) - \n(\vp - t)}\, ds dt.
\end{aligned}
\]
Taking absolute values, using the fact that $|\gamma'|$ is constant and  that the support of $\rho_{1/n}$ has length less than $2/n$, we obtain
\[
|\ip{\gamma_n',\n_n}(\vp)| \le 
\frac{\Len(\gamma)}{2\pi}  \frac{\omega_{\n}(4/n)}{|\rho_{1/n}*\n|(\vp)},
\]
where $\omega_\n$ is the modulus of continuity of $\n$. Letting $n\to\infty$, we obtain that
\[
\ip{\gamma_n',\n_n} \to 0 \qquad\text{uniformly}.
\]

The next step is to account for the fact that $\gamma_n'$ and $\n_n$ are orthogonal only asymptotically, hence $(\gamma_n,\n_n)$ is not a framed loop. 
We project the velocities $\gamma_n'$ onto the plane perpendicular to $\n_n$,
\[
\alpha_n = \gamma_n' - \ip{\gamma_n',\n_n} \n_n,
\]
obtaining $C^\infty(\Circ;\R^3)$ maps, converging to $\gamma'$ in $\Hhalf$ (once again by \lemref{lem:Hs_multiplication}),  bounded away from zero for $n$ large enough, and satisfying
\[
\alpha_n \perp \n_n.
\]
Consequently, 
\[
\sigma_n(\vp) = \int_0^{\vp} \alpha_n(s)\, ds
\]
is a smooth immersion $\Circ\setminus\{0\}\to\R^3$ satisfying $\sigma_n'\perp \n_n$, 
and converging to $\gamma$ in $\Hthreehalves(\Circ\setminus\{0\};\R^3)$. We would be done if $\sigma_n$ could be extended to a closed loop, however,
\[
\Delta_n = \sigma_n(2\pi) - \sigma_n(0) = \int_{\Circ} \brk{ \gamma_n' - \ip{\gamma_n',\n_n} \n_n}\, ds = 
- \int_{\Circ} \ip{\gamma_n',\n_n} \n_n\, ds
\]
is generally nonzero. Yet, 
\[
|\Delta_n| \le \int_{\Circ} |\ip{\gamma_n',\n_n}|\, ds \to 0
\]
by the uniform convergence of the integrand. 

Thus, we need to further approximated $\sigma_n$ by a map $\tsigma_n\in C^\infty_{\Imm}(\Circ;\R^3)$ that retains the orthogonality condition. We do it by 
modifying its velocity  $\alpha_n$ 
into a map $\talpha_n\in C^\infty(\Circ;\R^3)$, satisfying
\[
\talpha_n\perp\n_n
\qquad
\int_{\Circ} \talpha_n(s)\, ds = 0
\Textand
\talpha_n - \alpha_n\to 0 \quad\text{uniformly and in $\Hhalf$}.
\]

Suppose that the image of $\n$ spans $\R^3$. Then there exist smooth functions $\psi^1,\psi^2,\psi^3\in C^\infty(\Circ)$ such that
\[
v^i = \int_{\Circ} \psi^i(s) \n(s)\, ds \qquad i=1,2,3 
\]
are linearly-independent. For $i=1,2,3$,
\[
v^i_n = \int_{\Circ} \psi^i(s) \n_n(s)\, ds
\]
converge to $v^i$, and are linearly-independent for $n$ large enough. 
Let $c^1_n,c^2_n, c^3_n\in\R^3$ be the minimum-norm solutions to 
\[
\sum_{i=1}^3 c^i_n \times v^i_n = \Delta_n,
\]
which converge to zero. 
then
\[
\talpha_n = \alpha_n - \sum_{i=1}^3 \psi^i\, (c^i_n\times \n_n).
\]
satisfies the required conditions. 
If the span of $\n$ is two-dimensional, we adapt this construction noting that the span of the image of $\n_n$ is contained in the span of the image of $\n$.
If the span of $\n$ is one-dimensional, then it is constant, and thus $\n_n=\n$ and $\Delta_n=0$, so there is no need to correct $\sigma_n$.
\end{proof}

Finally, we will need the following:

\begin{lemma}[Stability of extendability under homotopy]
\label{lem:extendability_homotopy}
Let $\W\subset \R^2$ be a domain and let $f\in \HtwoImm(\W;\R^3)$ such that its normal $\n$ is in $\Hone(\W;\R^3)$.
If $\Gamma_1,\Gamma_2\subset \bW$ are two simple closed $C^{1,1}$ curves which are homotopic in $\W$, then $(f,\n)|_{\Gamma_1}$ is extendable if and only if $(f,\n)|_{\Gamma_2}$ is.
\end{lemma}

\begin{proof}
First, we note that since $\Gamma_i$ are $C^{1,1}$, we have that $f|_{\Gamma_i}\in \HthreehalvesImm(\Gamma_i;\R^3) \cong \HthreehalvesImm(\Circ;\R^3)$ and $\n\in \Hhalf(\Gamma_i;\Sph) \cong \Hhalf(\Circ;\Sph) $ \cite[Theorem~1.5.1.2]{Gri85}.
Since $\Gamma_1,\Gamma_2\subset \bW$ are homotopic in $\W$, they can be connected by a $C^{1,1}$ homotopy $(\Gamma_t)_{t\in[1,2]}$ (for smooth curves it is a standard result in differential topology that a continuous homotopy can be upgraded to a smooth one, whereas in our case we need first to mollify the $C^{1,1}$ curve).
By the continuity of the traces along the foliation (in local coordinates, this is equivalent to continuity of traces of $\R^2$ maps along parallel lines, which is, in turn, equivalent to the continuity of the composition with a translation operator on $\R^2$), we have that $t\mapsto (f,\n)|_{\Gamma_t}$ is a continuous map to $\Hthreehalves(S^1;\R^3)\times \Hhalf(S^1;\R^3)$.
Further, note that there exist $C>c>0$ such that $f|_{\Gamma_t}\in \Hthreehalves_{\Imm,c,C}(\Circ;\R^3)$ for all $t\in [1,2]$, where $\Hthreehalves_{\Imm,c,C}(\Circ;\R^3)$ is defined in \lemref{lem:H12frames_well_defined};
this holds because $f|_{\Gamma_t}$ are restrictions of the same map $f$.
Thus, we conclude from \lemref{lem:H12frames_well_defined} that the associated trihedra $(\frakt_1,\frakn_1,\n_1)$ and $(\frakt_2,\frakn_2,\n_2)$ lie in the same connected component of $\Hhalf(\Circ;\SO(3))$, which completes the proof.
\end{proof}

We finish this section by proving \propref{prop:Burgers_Hhalf}, implying that the definition of Burgers vectors extends to the case of $H^{1/2}$-framed loops.

\Emph{Proof of \propref{prop:Burgers_Hhalf}:}
Let $(\gamma,\n)$ be a smooth framed loop, and denote, for $\vp,\vp'\in \Circ$, $v_\vp(\vp') = \Pi^{\gamma,\n}_{\vp,\vp'} \gamma'(\vp)$. 
By \eqref{eq:parallel_transport}, we have that
\[
v'_\vp(t) = -\ip{v_\vp(t),\n'(t)}\, \n(t).
\]
hence
\[
v_\vp(0) = \gamma'(\vp) + \int_0^{\vp} \ip{v_\vp(t),\n'(t)}\,dt.
\]
By the definition of the Burgers vector \eqref{eq:def_burgers} and Fubini's theorem we thus obtain
\beq
\begin{aligned}
B_{\gamma,\n,0} 
&= \int_{\Circ} \gamma'(\vp) \,d\vp + \int_{\Circ} \int_0^\vp \ip{v_\vp(t),\n'(t)} \n(t) \, dt\, d\vp \\
&= 0 + \int_{\Circ}  \ip{\int_{t}^{2\pi} v_\vp(t)\, d\vp,\n'(t)}\n(t) \, dt\\
&= \int_{\Circ}  \ip{\calL(t),\n'(t)} \n(t) \, dt \\
&= \n'[\calL\otimes \n].
\end{aligned}
\label{eq:burgers_low_reg}
\eeq
This extends to $\Hhalf$-framed loops by density of smooth framed loops (\propref{prop:connected_framed_loops}), since in this case $\calL\otimes \n\in \Hhalf$ by \lemref{lem:Hs_multiplication}, and $\n'\in H^{-1/2}$.
\hfill\ding{110}

\section{Isoperimetric inequality for framed loops: Proof of \thmref{thm:GB+iso}}
\label{sec:GB+iso}

We start by considering isoperimetric inequalities on the sphere. 
Weiner \cite{Wei74} proved the following inequality, as a particular case of a broader theorem, based on techniques developed by Banchoff and Pohl \cite{BP71}:

\begin{theorem}[Weiner, 1974]
\label{th:WeinerChern}
Let $\gamma \in C^2_{\Imm}(\Circ;\Sph)$ be an immersed loop on the sphere and let $\Gamma = \gamma(\Circ)$. The total length $L$ of $\Gamma$ satisfies the isoperimetric inequality,
\[
L^2 \ge \frac{1}{2} \int_{\Sph\times\Sph} w_{\Gamma }^2\, \VolSph \wedge \,\VolSph,
\]  
where for $p,q\in\Sph$, $w_{\Gamma }(p,q)$ is the (signed, or directed) number of intersections of $\Gamma$ with a geodesic segment connecting $p$ and $q$, and $\VolSph$ is the volume element of the sphere. 
\end{theorem}

\[
\btkz
\node at (0,0) {\includegraphics[width=6cm]{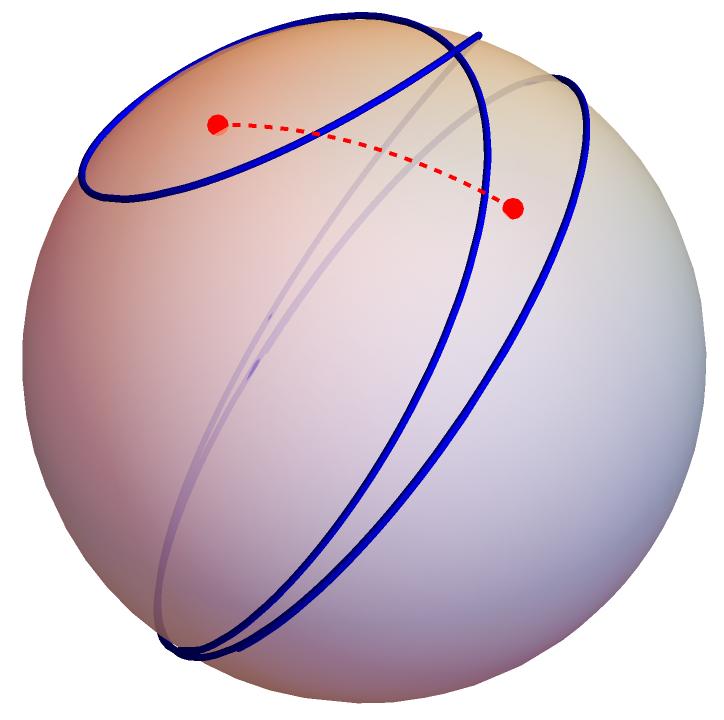}};
\node[text=blue] at (1,-1) {$\Gamma$};
\node[text=red] at (-1.1,2.2) {$p$};
\node[text=red] at (1.3,1.6) {$q$};

\etkz
\]

For $w_{\Gamma }(p,q)$, which Weiner refers to as a winding number, to be well-defined, the geodesic from $p$ to $q$ has to intersect $\Gamma$ transversely.  One can show that for rectifiable curves (and a fortiori for $C^2$-immersions), $w_{\Gamma}$ is well-defined on $\Sph\times\Sph$, except possibly on a set of measure zero, and is furthermore integrable (being piecewise constant). 
Likewise, since antipodal points  have measure zero in $\Sph\times\Sph$, we may define $w_{\Gamma}$ using the shortest geodesic. Alternatively, $w_\Gamma(p,q)$ can be defined everywhere on $(\SphG)\times(\SphG)$ by approximating $\Gamma$ uniformly by smooth curves nowhere tangent to the geodesic from $p$ to $q$, using the fact that $w_\Gamma$ is homotopy-invariant.

\thmref{th:WeinerChern} also holds for maps $\gamma\in C^2(\Circ;\Sph)$ that are not necessarily immersions. Indeed,  $C^2_{\Imm}(\Circ;\Sph)$ is dense in $C^2(\Circ;\Sph)$  \cite[Thm.~2.12]{Hir76}, hence the extension follows by approximation.

We proceed to use this isoperimetric inequality in the context of the Gauss map of a disc $D\subset\R^2$ immersed in $\R^3$. 
Let $f\in C^\infty_{\Imm}(\bD;\R^3)$ be an immersion with $\n\in C^\infty(\bD;\Sph)$ the corresponding Gauss map.
The curve $\Gamma = \n(\dD)$ partitions $\SphG$ into a finite number of connected components $\{\Lambda_i\}_{i\in I}$. For $p\in\SphG$, we denote by $\Lambda_p$ the connected component containing $p$. The topological degree of $\n$, denoted here $Q:\SphG\to\R$, can be defined via the pullback of volume forms \cite[p.~457]{Lee12},
\[
Q(p) = \int_D \n^*\eta,
\]
where $\eta$ is any smooth normalized 2-form supported in $\Lambda_p$; it assumes integer values and is constant on $\Lambda_p$ (the more standard notation for $Q(p)$ is $\deg(\n,D,p)$).
If $p$ is a regular value of $\n$, then $Q(p)$ coincides with the signed multiplicity of the image of $\n$ at $p$ \cite[Theorem.~17.35]{Lee12}. 

The degree $Q$ and winding number $w_\Gamma$ are related:  for $p,q\in\SphG$, we have
\beq
Q(p) - Q(q) = w_\Gamma(p,q).
\label{eq:Q_w}
\eeq
From this we obtain:

\begin{lemma}
\label{lem:7}
Let $f\in C^\infty_{\Imm}(\bD;\R^3)$.  
Then, the length $L$ of $\Gamma = \n(\dD)$ satisfies the isoperimetric inequality,
\[
L^2 \ge  4\pi \int_D (Q\circ\n)\, K\,\VolG - \brk{\int_D K\,\VolG}^2.
\]
\end{lemma}

\begin{proof}
Combining \thmref{th:WeinerChern} and \eqref{eq:Q_w} and using Fubini,
\[
\begin{aligned}
L^2 &\ge \frac{1}{2} \int_{\Sph\times\Sph} (Q(p) - Q(q))^2 \, \VolSph(p) \wedge \VolSph(q) \\
&= 4\pi \int_{\Sph\times\Sph} Q^2 \, \VolSph - \brk{\int_{\Sph} Q\,  \VolSph}^2.
\end{aligned}
\]
Next, for every bounded measurable function $h:\Sph\to\R$, 
\beq
\int_D (h\circ\n) K \,\VolG = \int_D (h\circ\n) \, \n^*\VolSph = \int_{\Sph} Q\, h\, \VolSph,
\label{eq:SverakSphere}
\eeq
where in the first equality we used the Gauss theorema egregium, $K \,\VolG = \n^*\VolSph$, and the second equality follows from the definition of $Q$ by approximating $h$ as an $L^1$ limit of sums of functions, each supported on a single $\Lambda_i$.
The claim now follows by applying \eqref{eq:SverakSphere}, once for $h=1$ and once for $h=Q$.
\end{proof}

\begin{lemma}
\label{lem:10}
The following inequality holds,
\[
\int_D (Q\circ\n)\, K\,\VolG \ge \Abs{\int_D K\,\VolG}.
\]
\end{lemma}

\begin{proof}
Writing $Q = \sum_i Q|_{\Lambda_i} \,\ind_{\Lambda_i}$ yields
\[
\int_D (Q\circ\n)\, K\,\VolG = \sum_{i\in I} Q|_{\Lambda_i} \int_{\n^{-1}(\Lambda_i)} K\,\VolG.
\]
Substituting $h = \ind_{\Lambda_i}$ in \eqref{eq:SverakSphere},
\beq
Q|_{\Lambda_i} |\Lambda_i| = \int_{\n^{-1}(\Lambda_i)} K \,\VolG.
\label{eq:Q_on_Lambda}
\eeq
Since $Q|_{\Lambda_i}$ is an integer, whose sign coincides by \eqref{eq:Q_on_Lambda} with the sign of the integral of $K$, it follows that
\[
\int_D (Q\circ\n)\, K\,\VolG \ge \sum_{i\in I} \Abs{\int_{\n^{-1}(\Lambda_i)} K\,\VolG}
\ge  \Abs{\sum_{i\in I} \int_{\n^{-1}(\Lambda_i)} K\,\VolG} = \Abs{\int_D K\,\VolG}.
\]
\end{proof}

Combining \propref{lem:7} and \lemref{lem:10} we obtain:

\begin{proposition}
\label{prop:BG+iso_disc}
Let $f\in C^\infty_{\Imm}(\bD;\R^3)$.  
Then,
\[
L^2 \ge \brk{4\pi -  \Abs{\int_D K\,\VolG}} \Abs{\int_D K\,\VolG},
\]
where $L$ is the length of $\Gamma = \n(\dD)$.
\end{proposition}

With that, we proceed to prove \thmref{thm:GB+iso}. The idea is to extend the boundary data $(\gamma,\n)$ into an immersion of a disc, using eventually \propref{prop:BG+iso_disc}.

\paragraph{Proof of \thmref{thm:GB+iso}} 
Since the result is trivial unless $\n \in W^{1,1}(\Circ;\Sph)$, we can assume that.
In view of  \lemref{lem:H32_W11_frames}, it suffices to prove the theorem for smooth framed loops, as the result follows then by approximation.

First, assume that $(\gamma,\n) $ is extendable,  and let $f\in C^\infty_{\Imm}(\bD;\R^3)$ be a smooth immersion such that $(f,\n)|_{\dD}$ coincides with the given framed loop. Denote by $\g$ the smooth induced metric on $D$ and by $K$ the corresponding Gaussian curvature. 
By the Gauss-Bonnet theorem, 
\[
\int_D K\,\VolG + \int_{\dD} \kg\,d\ell = 2\pi,
\]
Thus, we need to prove that
\[
\brk{\int_{\dD} |D_s\n|\, d\ell}^2 \ge 
\brk{4\pi - \Abs{\int_D K\,\VolG}} \Abs{\int_D K\,\VolG}.
\]
Noting that the integral on the left-hand side is
the length $L$ of $\Gamma=\n(\dD)$, this is precisely \propref{prop:BG+iso_disc}.

Assume now that $(\gamma,\n) $ is non-extendable. Without loss of generality, assume that $\gamma(0) = (1,0,0)$ and $\n(0) = (0,0,1)$.
We concatenate $(\gamma,\n)$  with $(\gamma_0,\n_0)$, as defined in \eqref{eq:trivial_frame}, to obtain a new, continuous piecewise smooth framed loop $(\tgamma,\tilde{\n})$; we denote by $\tilde{\kappa}_g$ and $d\tilde{\ell}$ the corresponding geodesic curvature and length form.
Since the moving trihedron corresponding to $(\gamma_0,\n_0)$ is a generator of the homotopy group $\bbZ/2\bbZ$, so is the moving trihedron corresponding to  $(\tgamma,\tilde{\n})$, and it follows from \propref{prop:connected_framed_loops} that $(\tgamma,\tilde{\n}) $ is extendable. 
Since we have already established the inequality \eqref{eq:GB+iso} for non-smooth extendable frames (in this case, piecewise-smooth),
we obtain
\[
\brk{\int_{\Circ} |D_s\tilde{\n}| d\tilde{\ell}}^2 \ge 
\brk{4\pi - \Abs{2\pi - \int_{\Circ} \tkg\, d\tilde{\ell}}} \Abs{2\pi - \int_{\Circ} \tkg\, d\tilde{\ell}}.
\]
Together with
\[
\int_{\Circ} |D_s\tilde{\n}|\, d\tilde{\ell} = \int_{\Circ} |D_s\n|\, d\ell 
\Textand
\int_{\Circ} \tkg\, d\tilde{\ell} = 2\pi + \int_{\Circ} \kg\, d\ell,
\]
the proof is complete.
\hfill$\blacksquare$

\section{Curvature estimate: Proof of \thmref{thm:main_theorem}}
\label{sec:main_theorem}

We start by showing that we may assume without loss of generality that the domain $\W$ is smooth, and replace Condition 1 in \thmref{thm:main_theorem} by $f\in C^1(\bW;\R^3)$. Indeed, suppose that the theorem holds under this extra-regularity condition, and let $\W$ be Lipschitz. 
Approximating $\W$ from the inside by smooth subdomains $U\Subset \W$ with the same topology\footnote{This is always possible for finitely-connected domains in the plane due to Koebe's circle domain theorem (see \cite{Bis07}).} (in which case $f\in C^1(\W;\R^3)$ implies that $f\in C^1(\bU;\R^3)$), the lower bound for the domains $U$ implies via dominated convergence the lower bound for the domain $\W$.

\subsection{The smooth case}
We start by proving the theorem for smooth immersions, namely:

\begin{quote}
Let $\W\subset\R^2$ be a finitely-connected, bounded, Lipschitz domain, and let $f\in C^\infty_{\Imm}(\bW;\R^3)$. 
If for every connected component $\Gamma\subset \dW$, the framed loop $(f,\n)|_{\Gamma} $ is extendable,
then
\[
\|d\n\|_{L^2(\W,\g)}^2 \ge \frac{4\pi - \|K\|_{\LoneMC}}{\|K\|_{\LoneMC}} \|K\|_{\HMoneMC}^2.
\]
\end{quote}

Here, $ \|K\|_{\LoneMC}$ and $\|K\|_{\HMoneMC}$ are as defined in \eqref{eq:LoneMC} and \eqref{eq:HMoneMC}.

\paragraph{The Riesz representer $u_K$}
Recall that we denote the connected components of $\dW$ by $\Gamma_i$, $i=0,\dots, m$, with $\Gamma_0$ being the outer boundary.
Let $f\in C^\infty_{\Imm}(\bW;\R^3)$ and denote, as above, by $\g$ and $K$ the pullback metric and the corresponding Gaussian curvature. 
For $m>0$, Let $K_i$ be the curvatures enclosed ``in the holes" as given by \eqref{eq:Ki}.
Assuming that $K$ and $K_i$ are not identically zero, 
there exists a unique test function $u^*\in \HoneMC$ maximizing  \eqref{eq:HMoneMC}, and satisfying $\|u^*\|_{\dotHnorm}=1$. 
A multiple of $u^*$, denoted $u_K\in C^\infty(\W)$ satisfies the elliptic boundary-value problem: 
\beq
\label{eq:ELPDE}
\begin{cases}
- \Delta_\g u_K=K 		&\quad\text{in }\,\, \Omega  \\
u_K = 0				&\quad\text{on }\,\, \Gamma_0 \\
u_K=\mathrm{constant} 	&\quad\text{on }\,\, \Gamma_i, \, i=1,\ldots,m, \\
\int_{\Gamma_i} \frac{\partial u_K}{\partial \nu }\, \iVolG = K_i&\quad\text{on }\,\, \Gamma_i, \\
\end{cases}
\eeq    
where $\Delta_\g$ is the Laplace-Beltrami operator, and $\nu$ is the outward-pointing normal; see Appendix \ref{sect:AppB}.
Such problems are known as \emph{floating potential problems}.
It follows immediately from \eqref{eq:HMoneMC} and \eqref{eq:ELPDE} that we have the dual norm relation
 \begin{equation}
\label{eq:dualnorm}
\|K\|_{\HMoneMC} = \left\| u_K \right\|_{\dotHnorm}. 
\end{equation}

As an immediate consequence of Sard's theorem and the preimage theorem \cite[p.~21]{GP74},
for almost every $\lambda$ in the range of $u_K$, the level set $\Glam = u_K^{-1}(\lambda)$ is a one-dimensional submanifold of $\W$.
Since $\W$ is compact, $\Glam$ is a finite disjoint union of $I(\lambda)\in\bbN$ closed smooth curves, 
\[
\Glam = \bigcup_{j=1}^{I(\lambda)}  \Glami.
\]
Denote by $\Wlami\subset\W$ the domain enclosed by $\Glami$.
For every regular value $\lambda$ and every $j$, 
\beq
\begin{aligned}
\int_{\Glami} |du_K|\,\iVolG &= 
\Abs{\int_{\Glami} \pd{u_K}{\nu}\,\iVolG} 
= \Abs{\int_{\Wlami}
\Delta_\g u_K \,\VolG - \sum_{\Gamma_i\subset \partial\Wlami} \int_{\Gamma_i} \frac{\partial u_K}{\partial \nu }\, \iVolG}  \\
&= \Abs{\int_{\Wlami} K \,\VolG + \sum_{\Gamma_i\subset \partial\Wlami} K_i}
\le \|K\|_{\LoneMC}.
\end{aligned}
\label{eq:use_many_times}
\eeq
The first equality follows from $\Glami$ being a level set of $u_K$, and the fact that the normal derivative of $u_K$ on $\Glami$ has constant sign.  
The second equality follows from Green's theorem. 
The third equality follows from \eqref{eq:ELPDE}  and the last inequality is immediate.

\paragraph{Completing the proof}
We complete the proof of \thmref{thm:main_theorem} for smooth immersions:
\[
\begin{aligned}
\|K\|_{\LoneMC}\, \|d\n\|_{L^2(\W)}^2 
&\overset{(*)}{=} \|K\|_{\LoneMC} \int_{\R} \brk{\sum_{j=1}^{I(\lambda)} \int_{\Glami} \frac{|d\n|^2}{|du_K|} \,\iVolG  }\, d\lambda \\
&\overset{\eqref{eq:use_many_times}}{\ge} \int_{\R} \sum_{j=1}^{I(\lambda)} \brk{ \int_{\Glami} |du_K| \,\iVolG  } \brk{\int_{\Glami} \frac{|d\n|^2}{|du_K|} \,\iVolG}\, d\lambda \\
&\overset{(**)}{\ge}  \int_{\R}  \sum_{j=1}^{I(\lambda)} \brk{\int_{\Glami} |d\n| \,\iVolG }^2\, d\lambda \\
 &\overset{(***)}{\ge} \int_{\R} \sum_{j=1}^{I(\lambda)} \brk{\brk{4\pi -  \Abs{ \int_{\Wlami} K\,\VolG + \sum_{\Gamma_i\subset \partial\Wlami} K_i}} \Abs{ \int_{\Wlami} K\,\VolG + \sum_{\Gamma_i\subset \partial\Wlami} K_i}} \, d\lambda \\
&\overset{\eqref{eq:use_many_times}}{\ge} \brk{4\pi - \|K\|_{\LoneMC}}  \int_{\R}  \brk{\sum_{j=1}^{I(\lambda)} \int_{\Glami} |du_K|\,\iVolG} \, d\lambda \\
&\overset{(*)}{=} \brk{4\pi - \|K\|_{\LoneMC}} \|du_K\|_{L^2(\W)}^2 \\
&\overset{\eqref{eq:dualnorm}}{=} \brk{4\pi - \|K\|_{\LoneMC}} \|K\|_{\HMoneMC}^2.
\end{aligned}
\]
The passages marked by (*) follow from the coarea formula \cite[Sec.~3.4]{EG15}, and (**) follows from the Cauchy-Schwarz inequality.

It remains to justify transition (***), that is, the inequality 
\beq
\brk{\int_{\Glami} |d\n| \,\iVolG }^2 \ge
\brk{4\pi -  \Abs{\int_{\Wlami} K\,\VolG + \sum_{\Gamma_i\subset \partial\Wlami} K_i}} \Abs{\int_{\Wlami} K\,\VolG + \sum_{\Gamma_i\subset \partial\Wlami} K_i}
\label{eq:justify_me}
\eeq
for almost every $\lambda$ in the range of $u_K$ and every $j$. Let $(\tgamma,\tilde{\n})$ denote the restrictions of $f$ and $\n$ to $\Glami$ and by $\tkg$ the corresponding geodesic curvature. Then, 
\[
|d\n| \,\iVolG \ge |D_s\tilde{\n}| \, d\ell
\qquad
\text{on $\Glami$},
\]
and by the Gauss-Bonnet theorem and the definition of $K_i$,
\[
2\pi - \int_{\Glami} \tkg  \,\iVolG = \int_{\Wlami} K\,\VolG + \sum_{\Gamma_i\subset \partial\Wlami} K_i.
\]
Eq.~\eqref{eq:justify_me} follows then from \thmref{thm:GB+iso}, along with the assumption that $f$ extendable, which by definition implies the extensibility of its restriction to every regular enough simple closed curve.  

\subsection{Lower regularity}

We now extend the result of the previous section to immersions of lower regularity.

\paragraph{Case 1: $f\in \HtwoImm(\W;\R^3) \cap C^1(\bW;\R^3)$} 
For this case, we need the following Lemma: 

\begin{lemma}\label{lem:immersion_convergence}
Let $f\in \HtwoImm(\W;\R^3) \cap C^1(\bW;\R^3)$ be an extendable immersion.
Let $f_n\in C^\infty(\bW;\R^3)$ be a sequence converging to $f$ both in $\Htwo(\W;\R^3)$ and in $C^1(\bW;\R^3)$. Denote by $\g$ and $\g_n$ the pullback metrics induced by $f$ and $f_n$, respectively, and by $\n$ and $\n_n$ the corresponding Gauss maps.
Then the following holds:
\begin{enumerate}[itemsep=0pt,label=(\alph*)]
\item Immersion property: $f_n\in C^\infty_{\Imm}(\bW;\R^3)$ for $n$ large enough.
\item Metric convergence: $\g_n \to \g$ and $\g_n^{-1} \to \g^{-1}$ in $\Hone(\W;\R^2\otimes\R^2)$ and in $C^0(\bW;\R^2\otimes\R^2)$.
\item Gauss map convergence: $\n_n\to \n$ in $\Hone(\W;\Sph)$ and in $C^0(\bW;\Sph)$.
\item Gaussian curvature convergence: $K_n\to K$ in $L^1(\W)$ and $K_{n,i}\to K_i$ for each $i\in I$.
\item Extendability: $f_n$ is extendable for $n$ large enough.
\end{enumerate}
\end{lemma}

Note that the $L^p$ convergence on tensor fields is metric dependent, both via the pointwise norms and via the volume forms. There is however no ambiguity as to whether we use the metrics $\g_n$ of $\g$, due to the uniform convergence of $\g_n$ and its inverse $\g_n^{-1}$ guaranteed by Item (b).

\begin{proof}
By assumption $df_n\to df$ in $\Hone(\W;\R^3)$ and in $C^0(\bW;\R^3)$.
\begin{enumerate}[itemsep=0pt,label=(\alph*)]
\item[(b)]
Since in two dimensions, the space $\Hone(\W)\cap C^0(\bW)$ is an algebra (closed under pointwise multiplication, and the multiplication is continuous), we obtain that $\g_n\to \g$ in $\Hone(\W;\R^2\otimes\R^2)$ and in $C^0(\bW;\R^2\otimes\R^2)$.
It follows that $\g_n^{-1}$ is bounded in $C^0$ for $n$ large enough and that $\g_n^{-1} \to \g^{-1}$ uniformly.  Since, in coordinates, $D\g^{-1} = -\g^{-1} (D\g) \g^{-1}$, and likewise for $\g_n$, it follows that $\g_n^{-1} \to \g^{-1}$ also in $\Hone(\W;\R^2\otimes\R^2)$. 

\item[(a)]
The properties of $\g_n$ and $\g_n^{-1}$ imply that $f_n$ are immersions in the sense of Definition~\ref{def:immersions} for $n$ large enough.

\item[(c)]
The coordinate expression for the normal $\n = \det(\g^{-1})\, \pl_1 f \times \pl_2 f$ is a product of functions in $\Hone(\W;\R^3)\cap C^0(\bW;\R^3)$ (and likewise for $\n_n$), hence $\n_n\to \n$ in $\Hone(\W;\R^3)$ and in $C^0(\bW;\R^3)$ by the same argument as for $\g_n$.

\item[(d)]
Since the Gauss equation holds for immersions in $\HtwoImm(\W;\R^3)\cap C^1(\bW;\R^3)$, i.e., the Gaussian curvature is the Jacobian determinant of the Gauss map (see, e.g., \cite[Proposition~5.2]{MM25}\footnote{The proof of \cite[Proposition~5.2]{MM25} assumes that the metric is smooth, yet it is sufficient to assume that $\g,\g^{-1} \in \Hone\cap L^\infty$, which follows from $f\in \HtwoImm(\W;\R^3)$, as this implies that the Christoffel symbols are in $L^2$, being products of $D\g$ and $\g^{-1}$.}), we obtain that $K_n\to K$ in $L^1(\W)$, as products of the $L^2$-converging derivatives of $\n_n$.
It remains to prove the convergence of the curvature enclosed in the holes.
Let $\tW_i \subset \W$ be an annulus bounded by $\Gamma_i$ and a smooth curve $\tGamma_i$.
Without loss of generality, by Fubini's theorem, we can choose $\tW_i$ such that $f_n\to f$ in $\Htwo(\tGamma_i;\R^3)$ and in $C^1(\tGamma_i;\R^3)$. 
This implies that the geodesic curvature of $\tGamma_i$ with respect to $\g_n$ converges in $L^2(\tGamma_i)$ to the geodesic curvature with respect to $\g$.
The convergence $K_{n,i}\to K_i$ follows from \eqref{eq:K_i_flexible} and the $L^1(\W)$ convergence $K_n\to K$.

\item[(e)]
The extendability of the frames at the boundary follows from the corresponding assumption on $f$, the $C^1(\bW;\R^3)$ convergence of $f_n\to f$ and the $C^0(\bW;\Sph)$ convergence $\n_n \to \n$.
Since the restrictions of $f_n$ and $\n_n$ to all the boundary components are extendable for $n$ large enough, then $f_n$ is extendable by 
the comment following \defref{def:W22_extendable}.
\end{enumerate}
\end{proof}

Equipped with \lemref{lem:immersion_convergence}, we prove \thmref{thm:main_theorem} for an extendable immersion $f\in \HtwoImm(\W;\R^3)\cap C^1(\bW;\R^3)$ by approximation: 
Let $f_n\in C^\infty(\bW;\R^3)$ converge to $f$ in $\Htwo(\W;\R^3)$ and in $C^1(\bW;\R^3)$.
Applying the smooth version of \thmref{thm:main_theorem} to $f_n$ (which is possible by clauses (a) and (e) in \lemref{lem:immersion_convergence}), we obtain
\[
\|d\n_n\|_{L^2(\W,\g_n)}^2 \ge \frac{4\pi - \|K_n\|_{L^1_\text{MC}(\W,\g_n)}}{\|K_n\|_{L^1_\text{MC}(\W,\g_n)}} \|K_n\|_{\Hminusone_\text{MC}(\W,\g_n)}^2.
\]
Assume that $\|K\|_{\LoneMC}<4\pi$, otherwise \thmref{thm:main_theorem} holds trivially.
Taking the limit $n\to \infty$ and using clauses (b), (c) and (d) in \lemref{lem:immersion_convergence}, we obtain that
\beq
\|d\n\|_{L^2(\W,\g)}^2 \ge \frac{4\pi - \|K\|_{\LoneMC}}{\|K\|_{\LoneMC}}\limsup_{n\to \infty} \|K_n\|_{\Hminusone_\text{MC}(\W,\g_n)}^2,
\label{eq:irregular_ineq1}
\eeq
where the need to take a superior limit of $\|K_n\|_{\Hminusone_\text{MC}(\W,\g_n)}$ as we do not a-priori know that it converges.
Since $\n\in \Hone(\W;\Sph)$, the left-hand side is finite, hence $\limsupn \|K_n\|_{\Hminusone_\text{MC}(\W,\g_n)}< \infty$.
By \lemref{lem:immersion_convergence}, $\g_n\to\g$ uniformly, and $K_{n,i}\to K_i$, and thus $K_n$ are uniformly bounded in $\Hminusone_\text{MC}$ with respect to a fixed metric $\g$ and fixed values $K_i$.
It follows that $K_n$ weakly converge in this space, and since $K_n \to K$ in $L^1$, it follows that $K_n \rightharpoonup K$ in $\HMoneMC$.
We further have that for every $\vp\in \HoneMC$,  
\[
\frac{\int_\W K \vp\,\VolG + \sum_i K_i \vp|_{\Gamma_i}}{\|\vp\|_{\dotHnorm}} = \limn
\frac{\int_\W K_n \vp\,\VolGn + \sum_i K_{n,i} \vp|_{\Gamma_i}}{\|\vp\|_{\dotHnormn}}
\le \liminfn \|K_n\|_{\Hminusone_\text{MC}(\W,\g_n)},
\]
from which we obtain that
\[
\|K\|_{\HMoneMC} \le \liminf_{n\to \infty} \|K_n\|_{\Hminusone_\text{MC}(\W,\g_n)},
\]
which together with \eqref{eq:irregular_ineq1} completes the proof.

\paragraph{Case 2: $f\in \HtwoImm(\W;\R^3)$ and $\g\in W^{2,p}(\W)$ for $p>1$}

Here, we need the following lemma:

\begin{lemma}\label{lem:approx_metrics}
Let $f\in \HtwoImm(\W;\R^3)$  be such that $\g\in W^{2,p}(\W)$ for $p\ge 1$, and let $\g_n\in C^\infty(\bW)$ be a sequence of metrics converging in $W^{2,p}$ to $\g$. 
Denote the Gaussian curvatures of $\g$ and $\g_n$ by $K$ and $K_n$, respectively.
Then, the following holds:
\begin{enumerate}[itemsep=0pt,label=(\alph*)]
\item $\g_n^{-1} \to \g^{-1}$ in $W^{2,p}(\W)$; in particular, $\g^{-1}\in W^{2,p}(\W)$.
\item $K_n \to K$ in $L^p(\W)$.
\end{enumerate}
Furthermore, for any smooth curve $\Gamma\subset \W$, the geodesic curvatures $\kappa_{\g_n}$ with respect to $\g_n$ converge to $\kappa_\g$ in $L^p(\Gamma)$. 
\end{lemma}

\begin{proof}
We prove the result for $p=1$. The general case is similar (and easier).
Since $W^{2,1}(\W) \inj C^0(\bW)$,
we obtain that $\g_n\to \g$ uniformly.
By the Definition~\ref{def:immersions} of an immersion and the uniform convergence, the eigenvalues of $\g_n$ are bounded away from zero, hence $\g_n^{-1}\to \g^{-1}$ uniformly as well. 
Combining the formulas 
\[
\begin{gathered}
D(\g^{-1}) = -\g^{-1} (D\g) \g^{-1} \\
D^2 (\g^{-1}) =  -\g^{-1} (D^2\g) \g^{-1} + 2 \g^{-1} (D\g) \g^{-1} (D\g) \g^{-1} 
\end{gathered}
\]
with the embedding $W^{2,1}(\W)\inj \Hone(\W) \cap C^0(\bW)$,
it follows that $\g_n^{-1}\to \g^{-1}$  in $W^{2,1}$.

Denote by $\Gamma$ and $\Gamma_n$ the Christoffel symbols of the second kind of $\g$ and  $\g_n$.
Since the entries of $\Gamma$ are product of components of $\g^{-1}$ and $D\g$,
it follows that $\Gamma_n\to \Gamma$ in $W^{1,1}$, using again the embedding $W^{2,1}(\W)\inj \Hone(\W)\cap C^0(\bW)$.
Finally, since 
$K \sim D\Gamma + \Gamma*\Gamma$ (i.e., $K$ is a sum of derivatives of $\Gamma$ and of products of $\Gamma$ with itself),  it follows that $K_n\to K$ in $L^1(\W)$.  

Finally, let $\Gamma\subset \W$ be a smooth curve, parametrized by $\gamma:[0,1]\to \W$.
By the trace theorem, $\g_n\to \g$ in $W^{1,1}(\Gamma)$, hence uniformly, and thus $|\gamma'|_{\g_n} \to |\gamma'|_{\g}$ in $W^{1,1}(0,1)$, and the normals to $\Gamma$ with respect to $\g_n$ converge to the normal with respect to $\g$ uniformly.
Thus $\kappa_{\g_n}\to \kappa_{\g}$ in $L^1$. 
\end{proof}

\begin{lemma}\label{lem:elliptic_regularity}
Let $f$, $\g$ and $\g_n$ by as in \lemref{lem:approx_metrics} for $p>1$.
Let $u_n$ be solutions of equation \eqref{eq:ELPDE} with respect to the metrics $\g_n$, curvatures $K_n$, and enclosed curvatures $K_i$ (independent of $n$).
Then $u_n$ are uniformly bounded in $L^\infty(\W)$.
\end{lemma}

\begin{proof}
Without loss of generality, we can assume that $p\in (1,2)$. 
As proved in Appendix~\ref{sect:AppB},
\[
\|u_n\|_{L^\infty(\W)} \le C \left( \left\| K_n \right\|_{L^p(\W,\g_n)} + \sum_{i=1}^m |K_i| \right),
\]
with a constant $C$ depending only on the $W^{2,p}$ norms of $\g_n$ and its inverse. By \lemref{lem:approx_metrics}, the right-hand side is uniformly bounded.
\end{proof}

We can now complete the proof of \thmref{thm:main_theorem}:
Let $f\in \HtwoImm(\W;\R^3)$ such that $\g\in W^{2,p}(\W)$ for $p>1$, and let $\g_n$ as in \lemref{lem:approx_metrics}.
Assume, without loss of generality that $\|\g_n-\g\|_{W^{2,p}}< 1/n$.
Since $\g\in W^{2,p}(\W)$ is a $C^0$ metric up to the boundary, we have that $\n\in \Hone(\W;\Sph)$ (by the same argument as in the proof of \lemref{lem:immersion_convergence}).
For a fixed $n$, let $\{\Glami\}_{\lambda \in \R; i\le I(\lambda)}$ be the level sets of the function $u_n$ as defined in \lemref{lem:elliptic_regularity}.

Following essentially the same steps as in the smooth case,
\[
\begin{aligned}
\|K_n\|_{L^1(\W,\g_n)}\, \||d\n|_\g\|_{L^2(\W,\g_n)}^2 &= \|K_n\|_{L^1(\W,\g_n)} \int_{\R} \brk{\sum_{j=1}^{I(\lambda)} \int_{\Glami} \frac{|d\n|_\g^2}{|du_n|_{\g_n}} \,d\Vol_{\iota^*\g_n}  }\, d\lambda \\
&\ge 
\int_{\R} \sum_{j=1}^{I(\lambda)} \brk{  \int_{\Glami} |du_n|_{\g_n} \,d\Vol_{\iota^*\g_n}   } \brk{  \int_{\Glami} \frac{|d\n|_\g^2}{|du_n|_{\g_n}} \,d\Vol_{\iota^*\g_n}   }\, d\lambda \\
&\ge  \int_{\R} \brk{ \sum_{j=1}^{I(\lambda)} \int_{\Glami} |d\n|_\g \,d\Vol_{\iota^*\g_n}  
 }  ^2\, d\lambda \\
 &\ge \beta_n\int_{\R} \brk{ \sum_{j=1}^{I(\lambda)} \int_{\Glami} |d\n|_\g \,\iVolG  
 }  ^2\, d\lambda \\
 &\overset{(*)}{\ge}  \beta_n\int_{\R} \sum_{j=1}^{I(\lambda)} \brk{\brk{4\pi -  \Abs{ \int_{\Wlami} K\,\VolG + \sum_{\Gamma_i\subset \partial\Wlami} K_i}} \Abs{ \int_{\Wlami} K\,\VolG + \sum_{\Gamma_i\subset \partial\Wlami} K_i} \, d\lambda} \\
 &\ge  \beta_n \brk{4\pi - \|K\|_{L^1(\W,\g)}} \int_{\R}  \sum_{j=1}^{I(\lambda)} \Abs{ \int_{\Wlami} K\,\VolG + \sum_{\Gamma_i\subset \partial\Wlami} K_i} \, d\lambda \\
 &\overset{(**)}{\ge} \beta_n \brk{4\pi - \|K\|_{L^1(\W,\g)}} \int_{\R}  \sum_{j=1}^{I(\lambda)} \Abs{ \int_{\Wlami} K_n\,d\Vol_{\g_n} + \sum_{\Gamma_i\subset \partial\Wlami} K_i} \, d\lambda + \alpha_n\\
&\ge \beta_n  \brk{4\pi - \|K\|_{L^1(\W,\g)}}  \int_{\R}  \brk{\sum_{j=1}^{I(\lambda)} \int_{\Glami} |du_n|_{\g_n}\,d\Vol_{\g_n}} \, d\lambda + \alpha_n\\
&= \beta_n\brk{4\pi - \|K\|_{L^1(\W,\g)}} \|du_n\|_{L^2(\W,\g_n)}^2 + \alpha_n \\
&= \beta_n\brk{4\pi - \|K\|_{L^1(\W,\g)}} \|K_n\|_{\Hminusone_{MC}(\W,\g_n)}^2 +\alpha_n, \\
\end{aligned}
\]
where $\alpha_n = O(n^{-1})$ and $\beta_n = 1-O(n^{-1})$.
Here, $ \|K_n\|_{\Hminusone_{MC}(\g_n)}$ refers to the $\Hminusone$-norm with respect to the $n$-independent enclosed curvatures $K_i$.
The arguments for the different passages are similar to the ones in the proof of the smooth case, with the following differences:
In the passage (*), we use that (a) for almost every $\lambda$, $\Glami$ are smooth closed curves (this was a primary reason for smoothing $\g$), and the restrictions of $f$ and $\n$ to $\Glami$ are in $\Htwo$ and $\Hone$, respectively, and (b) that the Gauss--Bonnet theorem holds for an immersion $f$ of this regularity (as follows, e.g., by approximation, using the convergence of $K_n$ and $\kappa_{\g_n}$ in \lemref{lem:approx_metrics}). 
In (**), we use the uniform boundedness of range of $u_n$, \lemref{lem:elliptic_regularity}, and the assumption on the rate of decay of $\|\g_n - \g\|_{W^{2,p}}$, which implies that $\|K_n - K\|_{L^1(\W,\g)} = O(1/n)$.

Taking $n\to \infty$ and using \lemref{lem:approx_metrics}, we obtain
\[
\|d\n\|_{L^2(\W,\g)} \ge \frac{4\pi - \|K\|_{L^1(\W,\g)}}{\|K\|_{L^1(\W,\g)}} \limsup_{n\to \infty} \|K_n\|_{\Hminusone_{MC}(\W,\g_n)}.
\]
The completion of the proof  follows the same argument as in the first case. 

\section{Burgers circuit estimate: Proof of \thmref{thm:burgers}}\label{sec:burgers}

Like for \thmref{thm:GB+iso}, it suffices to prove the assertion for smooth smooth framed loops, obtaining the Sobolev version via approximation (hinging on \lemref{lem:H32_W11_frames}).
By the third line in \eqref{eq:burgers_low_reg},
\[
\begin{aligned}
|B_{\gamma,\n,\vp_0}| &\le \|\calL\|_{L^\infty(\Circ)} \int_{\Circ}  |\n'(\vp)| \, d\vp 
\le \Len(\gamma(\Circ)) \, \int_{\Circ} |D_s\n|\, d\ell,
\end{aligned}
\]
where in the last passage we used the definitions of $D_s\n$ and $d\ell$, and the fact that for every $\vp'\in\Circ$, since parallel transport is norm-preserving,
\[
|\calL(\vp')| \le \int_{\Circ} |\Pi^{\gamma,\n}_{\vp,\vp'} \gamma'(\vp)|\, d\vp =  \int_{\Circ} |\gamma'(\vp)|\, d\vp  = \Len(\gamma(\Circ)).
\]

\section{Examples}\label{sec:examples}

\subsection{Cones: proof of \thmref{thm:cone}}\label{sec:cone}

A cone is a locally-flat surface whose curvature is concentrated at a single point---the apex. 
The curvature may be positive or negative; cones having positive curvature can be embedded in $\R^3$ as surfaces of revolution, unlike cones having negative curvature, which are known in the physics literature as \Emph{E-cones} (the E stands for \emph{excess}) \cite{MAG08}. 
 
Cones of either type can be parametrized using polar coordinates,
\[
(r,\vp) \in (0,\infty)\times\Circ,
\]
along with a metric of the form
\[
\g = dr\otimes dr + (\ca r)^2\, d\vp\otimes d\vp,
\]
where 
\[
\ca = 1 - \frac{\alpha}{2\pi}.
\]
The parameter $\alpha$ is the angular deficit/excess of the cone; it is in the range $(0,2\pi)$ for positively-curved cones, and negative for E-cones. 
A direct calculation via the Christoffel symbols shows that, as expected, $K=0$. On the other hand, let $\Gamma$ be any generator of the fundamental group, then 
\[
\int_{\Gamma} \kg\,\iVolG = 2\pi- \alpha,
\]
which shows that the curvature is $\delta$-measured at the origin, with intensity $\alpha$. 

We truncate the cone into a frustum by letting $r\in (r_0,R)$.
Thus, we obtain a locally-Euclidean annulus, $\W$, endowed with a locally-Euclidean metric $\g$. 
As before, we denote the outer boundary by $\Gamma_0$ and the inner boundary by $\Gamma_1$.
The curvature enclosed in the hole is $K_1 = \alpha$.

We proceed to solve the elliptic boundary-value problem \eqref{eq:ELPDE} for this geometry. 
Since the geometry is radially-symmetric, so is the solution $u_K$. 
The Laplace-Beltrami equation turns out not to depend on $\alpha$, and reads
\[
\frac{1}{r}\pd{}{r} \brk{r\, \pd{u_K}{r}} = 0,
\]
with boundary conditions
\[
u_K|_{\Gamma_0} = 0
\qquad
u_K|_{\Gamma_1} = \text{const}
\Textand
2\pi \ca r_0  \left.\pd{u_K}{r}\right|_{\Gamma_1} = - \alpha.
\]
The solution can be written down explicitly,
\[
u_K = -\frac{\alpha}{2\pi\ca} \log\frac{r}{R}.
\]
The relevant norms of $K$ are 
\[
\|K\|_{\LoneMC} = |\alpha|
\]
and 
\[
\|K\|_{\HMoneMC}^2 =  \|u_K\|_{\dotHnorm}^2 = 
\int_0^{2\pi}  \int_{r_0}^R  \brk{\frac{\alpha}{2\pi\ca r}}^2 \ca r\, dr d\vp 
= \frac{\alpha^2}{2\pi\ca} \log\frac{R}{r_0}.
\]
We can now apply \thmref{thm:main_theorem} to any extendable isometric immersion of $(\W,\g)$ and complete the proof of the first part of \thmref{thm:cone}.

We proceed to prove the second part of \thmref{thm:cone}.
Consider the loop $\gamma:[0,2\pi]\to\W$, given in coordinates by
\[
\gamma(t) = (r,t),
\]
for some $r\in(r_0,R)$.
Let its base point be $(r,0)$.
Using that 
\[
\begin{aligned}
e_1 &= \cos(\ca \vp) \, \partial_r + \sin(\ca \vp) \, (\ca r)^{-1}\partial_\vp \\
e_2 &= -\sin(\ca \vp) \, \partial_r + \cos(\ca \vp) \, (\ca r)^{-1}\partial_\vp
\end{aligned}
\]
are parallel orthonormal fields in $\{\vp\ne 0\}$,
a direct calculation shows that 
\[
\Pi^{\gamma,\n}_{t,0} \partial_\vp|_{(r,t)} = (\ca r) \sin(\ca t) \, \partial_r|_{(r,0)} + 
\cos(\ca t) \,  \partial_\vp|_{(r,0)}.
\]
The corresponding Burgers vector is given by
\[
\begin{aligned}
B_{\gamma,\n,0} &= \int_0^{2\pi} \Pi^{\gamma,\n}_{t,0} \gamma'(t)\, dt \\
&=  \int_0^{2\pi} \Pi^{\gamma,\n}_{t,0} \partial_\vp|_{(r,t)}\, dt \\
&= r (1 - \cos(2\pi \ca)) \partial_r|_{(r,0)} + \frac{1}{\ca} \sin (2\pi \ca) \partial_\vp|_{(r,0)}.
\end{aligned}
\]
Therefore, 
\[
|B_{\gamma,\n,0}| = 2r |\sin(\pi\ca)|,
\]
and thus, by \thmref{thm:burgers},
\[
\int_{\Circ} |D_s\n|\, d\ell \ge \frac{2r |\sin(\pi\ca)|}{2\pi\ca r} = \frac{|\sin(\pi\ca)|}{\pi\ca}.
\]
Thus, for any isometric immersion $f:\W\to \R^3$,
\beq\label{eq:cones_burgers_bound}
\begin{aligned}
\|d\n\|_{L^2(\W,\g)}^2 
&\ge \int_{r_0}^R \brk{\int_{\Circ} |D_s\n|^2_{\g} \,d\ell} dr 
\ge \int_{r_0}^R \frac{1}{2\pi \ca r} \brk{\int_{\Circ} |D_s\n|_{\g} \,d\ell}^2 dr \\
&\ge \int_{r_0}^R \frac{1}{2\pi \ca r} \frac{\sin^2(\pi\ca)}{(\pi\ca)^2} dr 
= \frac{\sin^2(\pi\ca)}{2\pi^3 \ca^3} \, \log \frac{R}{r_0},
\end{aligned}
\eeq
which completes the proof of the second part of \thmref{thm:cone}.

\begin{comment}
Whenever $\alpha = -2\pi k$ for some $k\in \bbN$, there exists an isometric immersion with zero bending: 
\[
f(r,\vp) = r(\cos(k+1)\vp,\sin(k+1)\vp,0).
\]
This immersion is extendable for even $k$, but not for odd $k$.
This shows that the extendability assumption in \thmref{thm:main_theorem} is essential, otherwise we would have obtained a contradiction to \eqref{eq:cone_lb_curvature} for $k=1$.
It also shows, for $k=2$, that the $4\pi - \|K\|_{\LoneMC}$ prefactor in \thmref{thm:main_theorem} cannot be relieved.
The same example, restricted to a fixed $r$, shows similarly the tightness of \thmref{thm:GB+iso} in the extendable and non-extendable cases.
\end{comment}

\subsection{Cone dipoles (dislocations): proof of \thmref{thm:LB_disloc_optimal}}\label{sec:dislocations}

Consider now a cone dipole.
As discussed in Section~\ref{subsec:examples}, \thmref{thm:main_theorem} 
yields a significantly sub-optimal lower bound in this case.
The body we consider here is as in Figure~\ref{fig:dislocation}, after cutting out a region of diameter $O(r_0)$ containing the two holes (i.e., the domain enclosed by the curve $\sigma$ in Figure~\ref{fig:dislocation}).

As proved in \cite[Sec.~3.1]{KM25}, the magnitude of the dislocation and the dimensions $r_0$ and $R$ essentially determine a unique geometry, modulo the precise shape of the inner and outer boundaries, which are immaterial for the scaling of the lower bound with $r_0$ and $R$.
Hence, we adopt a parametrization that is the easiest to perform calculation with.
Specifically, we fix $\e\in (0,r_0)$, and let  
\[
\W = B_R\setminus B_{r_0},
\]
endowed with the metric
\beq
\g_{rr} =1
\qquad
\g_{r\vp} =  0  
\Textand
\g_{\vp\vp} = \brk{r+\frac{\e}{\pi}\cos \vp}^2 
\label{eq:metric_disloc}
\eeq
where $(r,\vp)$ are the polar coordinates.

A direct calculation, similar to \cite[Appendix A]{Kup17} shows that the magnitude of the Burgers vector of this body is $\e$,
for every generator $\gamma$ of the fundamental group of the annulus, and in particular, for every constant-$r$ curve.
Since from \eqref{eq:metric_disloc}, the perimeter of a constant-$r$ curve $\sigma_r$ is 
$2\pi r $,
it follows from \thmref{thm:burgers} that
\beq\label{eq:disloc_naive_lb_leaf}
\int_{\sigma_r} |D_s \n|\, d\ell \ge \frac{\e}{2\pi r}
\eeq
(cf., \cite[Eq.~8]{Kup17}).
Foliating the domain with constant $r$-curves and using this lower bound as in Section~\ref{sec:cone} only yields the bound \eqref{eq:dislocation_naive_lb} that does not diverge with $R/r_0\to \infty$.

This failure is due to fact that the infimal bending energy of this manifold is significantly higher than the sum of the infimal bending energies of each leaf in the foliation.
This is a manifestation of the super-additivity property mentioned in the introduction, which in practice manifests in the following manner: 
if a surface with a dislocation having minimal bending energy is dissected into annuli, each annulus will relax into a lower-energy configuration.
We now overcome this issue and prove \thmref{thm:LB_disloc_optimal}.

First we note that we are considering $H^2$-isometric immersions of a smooth, flat metric. 
Since smooth isometric immersions of flat metrics are dense in the $H^2$-isometric immersions \cite{Hor11}, it is sufficient to consider smooth immersions.

Every smooth, locally-Euclidean submanifold of $\R^3$ is developable \cite{Mas62}:
It can be partitioned into flat points, at which the mean curvature $H$ equals zero, and non-flat points, the latter constituting an open set.
Through every non-flat point passes a unique asymptotic line, which is a geodesic (in the surface) whose image in $\R^3$ is a straight line (i.e., a geodesic in the ambient space). 
Asymptotic lines do not intersect and do not terminate until hitting the boundary \cite[Theorem~II]{Mas62}. 
Furthermore, since the Gaussian curvature is identically zero, $2 |H| = |d\n|$. 

Since the inner boundary in our case is geodesically convex, asymptotic lines in $\W$ either connect the inner boundary $\Gamma_1$ to the outer boundary $\Gamma_0$, or, connect the outer boundary to itself.  Asymptotic lines that connect the outer boundary to itself create folds that can be flattened out without affecting the rest of the surface, thus reducing the bending energy. Hence, for the sake of finding lower bounds, we may assume that every asymptotic line emanating from the outer boundary intersects eventually the inner boundary.
Furthermore, consider the marginal case where an asymptotic line emanates from the inner boundary, $\{r = r_0\}$, tangentially.
This asymptotic line will intersect the set $\{r = 2r_0\}$ at an angle of $\pi/6$. Thus, by restricting our attention to $r\in(2r_0,R)$, we may assume that all the asymptotic lines intersect constant-$r$ curves at an angle that is at least $\pi/6$.

Let $X$ be a unit vector field in $\W$, whose flow is from the inner boundary to the outer boundary, and  is parallel to asymptotic lines in non-flat points.

\begin{lemma}
\label{lem:dev_surf}
The mean curvature flux is conserved:
\[
\div(HX) = 0,
\]
from which follows that
\beq\label{eq:Lie_der_of_H}
d\iota_X  (H\, \VolG) = 0.
\eeq
\end{lemma}

\begin{proof}
Let $Y$ be a vector field such that $\{X,Y\}$ is an orthonormal frame.
By the definition of the asymptotic lines and the definition of the mean curvature as half the trace of the shape operator $S$, we have
\[
S(X) = 0
\Textand S(Y) = 2H\, Y.
\]
The divergence of $X$ is given by
\beq\label{eq:divX}
\div X = \ip{\nabla_X X, X} + \ip{\nabla_Y X, Y} = \ip{\nabla_Y X, Y},
\eeq
where the first term vanishes because $X$ is a unit vector field. 

The Codazzi equations are
\[
(\nabla_X S)(Y) = (\nabla_Y S)(X),
\]
or equivalently,
\[
\nabla_X (S(Y)) - S(\nabla_X Y) = \nabla_Y(S(X)) - S(\nabla_Y X),
\]
which reduces to
\[
2 X(H)\, Y  + 2H\, \nabla_X Y  - S(\nabla_X Y) + S(\nabla_Y X) = 0.
\]
Take an inner product with $Y$, using the fact that $S$ is symmetric and once again the identity $S(Y) = 2HY$,
\[
2 X(H) + 2H \ip{\nabla_Y X,Y} = 0.
\]
Combining with  \eqref{eq:divX},
\[
\div(HX) = X(H) + H \,\div(X) = 0.
\]
Finally  \cite[p. 32]{Lee18},
\[
0 = \div(HX)\,\VolG = d\iota_{HX} \VolG = d\iota_X  (H\, \VolG).
\]
\end{proof}

Using the fact that $H$ does not change sign along asymptotic lines, as well as the identity $2 |H| = |d\n|$,  it follows from 
\eqref{eq:Lie_der_of_H} that 
\beq\label{eq:Lie_der_of_H_Catan}
d \iota_X (|d\n|\,\VolG) = 0.
\eeq
Let $\sigma\subset\W$ be any loop homotopic to $\Gamma_1$.
Integrating \eqref{eq:Lie_der_of_H_Catan} over the domain bounded by $\Gamma_1$ and $\sigma$, we obtain by Stokes' theorem
\beq\label{eq:dN_sigma_Gamma}
\int_{\sigma} |d\n|\, \ip{X,\nu}\, \iVolG = \int_{\Gamma_1} |d\n|\, \ip{X,\nu}\, \iVolG,
\eeq
where $\nu$ is the normal to the curve. 

Denote by $\sigma_r$ the constant-$r$ loops.
By the above discussion of the intersection of asymptotic line with constant-$r$ curves, we always have
\[
\ip{X,\nu} \ge \cos\frac{\pi}{6} = \frac{1}{\sqrt{2}},
\]
hence for
every $r\in(2 r_0,R)$,
\beq
\begin{aligned}
\int_{\sigma_r} |d\n|\, \iVolG &\ge \int_{\sigma_r} |d\n|\, \ip{X,\frakn}\, \iVolG 
\overset{\eqref{eq:dN_sigma_Gamma}}{=}  \int_{\sigma_{2r_0}} |d\n|\, \ip{X,\frakn}\, \iVolG 
\ge \frac{1}{\sqrt{2}}\, \int_{\sigma_{2r_0}} |d\n|\,  \iVolG \\
&\overset{\eqref{eq:disloc_naive_lb_leaf}} \ge \frac{1}{4\sqrt{2}\pi} \frac{\e}{r_0}.
\end{aligned}
\label{eq:sigma'}
\eeq
Note the improvement of this bound over \eqref{eq:disloc_naive_lb_leaf}, as it does not deteriorate with $r$.

We proceed now as in \eqref{eq:cones_burgers_bound}: 
\[
\|d\n\|_{L^2(\W,\g)}^2   
\ge \int_{2r_0}^R \brk{\int_{\Circ} |d\n|^2_{\g} \,d\ell} dr 
\ge \int_{2 r_0}^R \frac{1}{2\pi r} \brk{\int_{\Circ} |d\n|_{\g} \,d\ell}^2 dr
\overset{\eqref{eq:sigma'}}{\ge} \frac{1}{64\pi^3}\frac{ \e^2}{r_0^2} \log \frac{R}{r_0}.
\]
This completes the proof of \thmref{thm:LB_disloc_optimal}.

\appendix

\section{Multiplication and composition in $\Hs$}

In this appendix, we prove some technical results concerning functions with $\Hs$ regularity, used in Section \ref{sec:Darboux}.
We begin by proving that the composition of an $\Hs$ function with a Lipschitz function is well-behaved. 

\begin{lemma}\label{lem:Hs_composition1}
Let $F:\R^k\to \R^m$ be a Lipschitz map, and let $\W\subset \R^d$ be a bounded Lipschitz domain (or a compact manifold).
Then, for any $s\in (0,1)$, the map $f\mapsto F\circ f$ is a continuous map from $\Hs(\W;\R^k)$ to $\Hs(\W;\R^m)$.
\end{lemma}

\begin{proof}
\textbf{Step 1: Uniform bounds and weak convergence.}
Let $\| F \|_{\mathrm{Lip}}$ denotes the Lipschitz constant of $F$.
From the definition of the Gagliardo $\Hs$ seminorm, it is immediate that $ \| F\circ f \|_{\dHs}^2 \le  \| F \|_{\mathrm{Lip}}^2 \| f \|_{\dHs}^2$, hence the composition is well-defined.
To prove continuity, let $f_n$ be a sequence in $\Hs(\W;\R^k)$ such that $f_n \to f$ strongly. Then $F \circ f_n$ is uniformly bounded in $\Hs(\W;\R^k)$, and by uniqueness of limits converges to $F \circ f$ weakly in $\Hs(\W;\R^k)$.

\textbf{Step 2: Strong convergence. }
To prove that this convergence is in fact strong, we will prove that $ \lim_{n \to \infty} \| F\circ f_n \|_{\Hs} = \| F\circ f \|_{\Hs} $.
As the weak $\Hs$ convergence $F\circ f_n\weakly F\circ f$ implies strong $L^2$ convergence, it remains to prove convergence of the $\dHs$-seminorm.
To prove this claim we will split the domain into points that are close together and far apart. 
Let $\epsilon >0$, then there holds 
\beq
\begin{aligned}
\limsup_{n \to \infty} \| F\circ f_n \|_{\dHs}^2
& = \limsup_{n \to \infty}  \iint_{  |x-y| > \epsilon} \frac{ \left| F(f_n(x)) - F(f_n(y)) \right|^2}{|x-y|^{d + 2s }} dx dy \\
&\qquad + \iint_{  |x-y| \le \epsilon} \frac{ \left| F(f_n(x)) - F(f_n(y)) \right|^2}{|x-y|^{d + 2s}} dx dy  \\
&\le  \limsup_{n \to \infty}  \iint_{ |x-y| > \epsilon} \frac{ \left| F(f_n(x)) - F(f_n(y)) \right|^2}{|x-y|^{d + 2s}} dx dy \\
&\qquad + \limsup_{n \to \infty}  \| F \|_{\mathrm{Lip}} \iint_{ |x-y| \le \epsilon} \frac{ \left| f_n(x) - f_n(y) \right|^2}{|x-y|^{d + 2s}} dx dy  \\
&= \iint_{ |x-y| > \epsilon} \frac{ \left| F(f(x)) - F(f(y)) \right|^2}{|x-y|^{d + 2s}} dx dy \\
&\qquad + \| F \|_{\mathrm{Lip}} \iint_{ |x-y| \le \epsilon} \frac{ \left| f(x) - f(y) \right|^2}{|x-y|^{d + 2s}} dx dy \\
& \le \|F\circ f\|^2_{\dHs}+ \| F \|_{\mathrm{Lip}} \iint_{ |x-y| \le \epsilon} \frac{ \left| f(x) - f(y) \right|^2}{|x-y|^{d + 2s}} dx dy, 
\end{aligned}
\label{eq:longeq}
\eeq 
where in the third passage, the limit of the first addend follows from the $L^2$ convergence $F\circ f_n\to F\circ f$, and in the second addend from the strong $H^s$ convergence $f_n\to f$.
Letting $\epsilon\to0$ we obtain that 
\[
\begin{aligned}
\limsup_{n \to \infty} \| F\circ f_n \|_{\dHs}^2 \le \| F\circ f \|_{\dHs}^2 . 
\end{aligned}
\]
Along with weak lower semicontinuity of the $\Hs$ norm, this concludes the proof. 
\end{proof}

We move on to the second lemma of this section, which concerns the composition of $\Hs$ functions. 

\begin{lemma}\label{lem:Hs_composition2}
Let $\W, \W'$ be bounded Lipschitz domains (or compact Riemannian manifolds) and let $F:\W' \to \W$ be a bi-Lipschitz map.
Then, for any $s\in (0,1)$, the map $f \mapsto f\circ F$ is a continuous linear map from $\Hs(\W;\R^d)$ to $\Hs(\W';\R^d)$.
Moreover, for any constant $L>0$, let $ H^{\sigma}_L(\W';\W) := \{F\in H^\sigma(\W';\W) ~:~ \text{$F$ is $L$-bi-Lipschitz}\}$. Then the action
\[
\Hs(\W;\R^d) \times H^{\sigma}_L(\W';\W) \to \Hs(\W';\R^d), \quad (f,F)\mapsto f\circ F
\]
is continuous for any $s\in (0,1)$ and $\sigma \ge 1$.
\end{lemma}

\begin{proof}
Changing variables, 
\[
 \| {f} \circ F \|_{\dHs}^2 = \iint_{ \W \times \W}  \left| f(x) - f(y) \right|^2 K_F(x,y) dx dy,
\]
where
\[
K_F(x,y) = \frac{\left| \det \nabla F (F^{-1}(x)) \right|^{-1} \left| \det\nabla F (F^{-1}(y)) \right|^{-1}}{|F^{-1}(x)-F^{-1}(y)|^{d + 2s}}. 
\]
For an $L$-biLipschitz map, these exists a constant $C(L)$, such that
\beq\label{eq:kernel_bound}
K_F(x,y) \le \frac{C(L)}{|x-y|^{d+2s}},
\eeq
Thus $\| {f} \circ F \|_{\dHs}^2 \le C \| f \|_{\dHs}^2$ and the map $f\mapsto f\circ F$ is well-defined.
Furthermore, as in \lemref{lem:Hs_composition1}, it is continuous with respect to weak convergence, noting that weak convergence $F_n\to F$ in $H^{\sigma}_L(\W';\W)$ implies uniform convergence of $F_n$ and $F_n^{-1}$ by the uniform bi-Lipschitz bound.

To complete the proof, it suffices to show, as before, that for $f_n\to f$ in $H^s$ and $F_n\to F$ in $H^\sigma$ we have that $\limsup \|f_n \circ F_n\|_{\dHs} \le \|f \circ F\|_{\dHs}$.
Since $F_n\to F$ in $H^\sigma$ and $\sigma\ge 1$, we can assume that $\nabla F_n\to \nabla F$ almost everywhere; since we already know that $F_n^{-1}\to F^{-1}$ uniformly, we have that $K_{F_n}\to K_F$ almost everywhere.
Thus we obtain that $\|f\circ F_n\|_{\dHs}\to \|f\circ F\|_{\dHs}$ by the dominated convergence theorem, using \eqref{eq:kernel_bound}, and thus
\[
\|f_n \circ F_n\|_{\dHs} \le \|f\circ F_n\|_{\dHs} + \|(f_n -f) F_n\|_{\dHs} \le  \|f\circ F_n\|_{\dHs} + C\|f_n -f\|_{\dHs} \to \|f\circ F\|_{\dHs},
\]
which completes the proof.
\end{proof}

The last lemma in this section shows that the multiplication of $\Hs$ functions is well-behaved:

\begin{lemma}\label{lem:Hs_multiplication}
Fix $B>0$, $s\in (0,1)$ and $\W\subset \R^d$ be a bounded Lipschitz domain. 
Then the pointwise multiplication map $(f_1,f_2)\mapsto f_1f_2$ is a continuous map $\Hs_B(\W)\times \Hs_B(\W)\to \Hs(\W)$ with respect to the $\Hs$ topology, where $\Hs_B(\W) = \{f\in \Hs(\W) ~:~ \|f\|_{L^\infty} \le B\}$.
\end{lemma}

\begin{proof}
\textbf{Step 1: Uniform bounds and weak convergence. }
It follows immediately from the definition of the Gagliardo seminorm that $ \| f_1f_2 \|_{\dHs} \le C \left(  \| f_1 \|_{L^\infty} \| f_2 \|_{\dHs}+  \| f_2 \|_{L^\infty} \| f_1 \|_{\dHs} \right)$ for some uniform constant $C>0$, and similarly for the $L^2$ part of the norm.
Thus the multiplication is well-defined.

In order to prove continuity, we proceed as in  \lemref{lem:Hs_composition1}. Let $f_{1}^{n}, f_{2}^{n}$ be two sequences in $\Hs_B(\W)$ such that $f^{n}_{1} \to f_{1}$ and $f^{n}_{2} \to f_{2}$ strongly in $\Hs$. Then $f_{1}^{n}f_{2}^{n}$ is uniformly bounded in the $\Hs$ norm, and by uniqueness of limits converges to $f_{1}f_{2}$ weakly in $\Hs_B(\W)$.

\textbf{Step 2: Strong convergence.}
As in \lemref{lem:Hs_composition1}, it is sufficient to prove that $ \lim_{n \to \infty} \| f_{1}^{n}f_{2}^{n} \|_{\dHs} = \| f_{1}f_{2} \|_{\dHs} $.
The proof proceed similarly, splitting the integral into $\{|x-y|>\e\}$ and $\{|x-y|<\e\}$ for some $\e>0$, where the convergence away from the diagonal follows from the $L^2$ convergence $f_1^nf_2^n \to f_1f_2$, and the integrand close to the boundary is estimated similarly and shown to vanish as $\e\to 0$.
\end{proof}

\section{Elliptic theory}
\label{sect:AppB}

In this appendix, we present some results regarding the elliptic equation \eqref{eq:ELPDE}. 
We begin by proving existence and uniqueness of an optimal test function in the definition of the $\HMoneMC$ norm. 

\begin{lemma}
Let $\W\subset\R^2$ be a finitely-connected, Lipschitz domain.
Denote by $\Gamma_0 = \partial D$ and $\Gamma_i = \partial D_i$, $i=1,\dots,m$ the outer and inner boundaries.
Let $\g$ be a bounded Riemannian metric on $\W$ with bounded inverse. 
Let $K\in L^p$ and $K_i \in \R$ for $i=1,\dots,m$, not all of them zero.
Consider the functional $\mathcal{E}: \HoneMC \to \R$ defined as 
\[
\mathcal{E} (u):= \int_\W u K\, \VolG + \sum_{i=1}^m K_i u|_{\Gamma_i}. 
\] 
Then there exists a unique $u^* \in \HoneMC$ such that $\left\| u^* \right\|_{\dotHnorm} = 1$ and 
 \[
\mathcal{E} (u^*) = \sup_{\left\| u\right\|_{\dotHnorm} = 1} \mathcal{E} (u).
\]
\end{lemma}

\begin{proof}
Existence follows by standard use of the direct method, and the uniqueness follows from the strict convexity of the unit ball in the $\HoneMC$ norm.
The assumption that not all the curvatures are zero is used in the uniqueness.
\end{proof}

\begin{remark}
For sufficiently regular $\g$, the Euler-Lagrange equations of $\mathcal{E}$ are given by \eqref{eq:ELPDE}, whose solution yields a multiple of the norm-$1$ maximizer $u^*$ satisfying \eqref{eq:dualnorm}.
\end{remark}

\bigskip

Throughout this appendix we denote by $\lesssim$ inequalities modulo a constant that only depends on $\W$ and on the $W^{2,p}$ norms of $\g$ and its inverse. 
\lemref{lem:elliptic_regularity} requires the following estimate: 
let $u$ be a solution of the boundary-value problem
\beq
\label{eq:the_BVP}
-\Delta_\g u = K 
\qquad
u|_{\Gamma_0} = 0 
\qquad
u|_{\Gamma_i} = \text{constant}
\qquad
\oint_{\Gamma_i} \pd{u}{\nu}\,\iVolG = K_i
\eeq
for $\g\in W^{2,p}$, $p\in (1,2)$. Then, 
\beq
\|u\|_{L^\infty(\W)} \lesssim \|K\|_{L^p(\W,\g)} + \sum_{i=1}^m |K_i|.
\label{eq:the_stuff_we_really_need}
\eeq
Throughout this appendix, we will use the notation 
\[
\Kbnd := \|K\|_{L^p(\W,\g)} + \sum_{i=1}^m |K_i|.
\]

The rest of this appendix is dedicated to proving equation \eqref{eq:the_stuff_we_really_need}, which is the content of Proposition~\ref{prop:whatwereallyneed}, building upon some preliminary lemmas.
Throughout below, denote by
\[
c_i = u|_{\Gamma_i} \qquad i=1,\dots,m
\]
the constant values of $u$ on the inner boundaries.

Since the $L^\infty$-norm is oblivious to the Riemannian metric, and since most of the literature concerns Euclidean domains, we carry out most of the analysis in a Euclidean setting, denoting the Euclidean metric by $\euc$. In particular, the Hessian $\nabla^2$ denotes the matrix of second derivatives with respect to Euclidean coordinates, rather than the covariant Hessian $\nabla^\g d$. The Sobolev spaces with respect to the two metrics are identical and their corresponding norms are equivalent:
\beq
\|\cdot\|_{W^{k,p}(\W,\euc)} \lesssim \|\cdot\|_{W^{k,p}(\W,\g)} \lesssim \|\cdot\|_{W^{k,p}(\W,\euc)},
\label{eq:equiv_Wkp}
\eeq
and similarly for norms on $\dW$.
We use repeatedly the following inequalities:
\begin{enumerate}[itemsep=0pt,label=(\alph*)]
\item
Poincar\'e inequality: for every $\vp\in\HoneMC$:
\beq
\label{eq:poincare}
\|\vp\|_{L^2(\W,\g)} 
\EqNo{\lesssim}{eq:equiv_Wkp} 
\|\vp\|_{L^2(\W,\euc)} 
\lesssim 
\|\nabla \vp\|_{L^2(\W,\euc)} 
\EqNo{\lesssim}{eq:equiv_Wkp} 
\|\nabla \vp\|_{L^2(\W,\g)}.
\eeq
In particular,
\beq
\label{eq:H1_H1dot}
\|\vp\|_{H^1(\W,\g)} \lesssim \|\vp\|_{\dotHnorm}.
\eeq

\item
Sobolev embedding theorem:  for every $q>1$, and for every $\vp\in\HoneMC$:
\beq
\label{eq:sob_emb}
\|\vp\|_{L^q(\W,\g)} \EqNo{\lesssim}{eq:equiv_Wkp} 
\|\vp\|_{L^q(\W,\euc)} \lesssim \|\vp\|_{H^1(\W,\euc)} 
\EqNo{\lesssim}{eq:equiv_Wkp} \|\vp\|_{H^1(\W,\g)}
\EqNo{\lesssim}{eq:H1_H1dot} \|\vp\|_{\dotHnorm}.
\eeq

\item
Trace inequality: for every function $\vp\in\HoneMC$:
\beq
\label{eq:trace_ineq}
\|\vp|_{\dW}\|_{L^2(\dW,\g)} 
\EqNo{\lesssim}{eq:equiv_Wkp} 
\|\vp|_{\dW}\|_{L^2(\dW,\euc)} 
\lesssim 
\|\vp\|_{H^1(\W,\euc)} 
\EqNo{\lesssim}{eq:sob_emb} 
\|\vp\|_{\dotHnorm}.
\eeq

\end{enumerate}

In the following (with the exception of Lemma \ref{lem:regularity}), we consider the following setting: 
$\W\subset\R^2$ is a finitely-connected, Lipschitz domain, whose outer and inner boundaries are denoted by $\Gamma_0 = \partial D$ and $\Gamma_i = \partial D_i$, $i=1,\dots,m$, respectively.
$\g$ is a Riemannian metric on $\W$ such that $\g \in W^{2,p} (\W, \g)$ for some $p>1$. 
Furthermore we fix $K\in L^p$ and $K_i \in \R$ for $i=1,\dots,m$.
\begin{lemma}
\label{lem:Kmins1_Lp}
\beq
\label{eq:combine6}
\|K\|_{\HMoneMC} \lesssim \Kbnd.
\eeq
\end{lemma}

\begin{proof}
By definition \eqref{eq:HMoneMC},
\[
\|K\|_{\HMoneMC} = \sup_{\vp\in\HoneMC} \frac{\int_\W K\vp \,\VolG + \sum_{i=1}^m K_i \vp|_{\Gamma_i}}{\|\vp\|_{\dotHnorm}} \le \sup_{\vp\in\HoneMC} \frac{\int_\W K\vp \,\VolG }{\|\vp\|_{\dotHnorm}} +\sup_{\vp\in\HoneMC} \frac{ \sum_{i=1}^m K_i \vp|_{\Gamma_i}}{\|\vp\|_{\dotHnorm}}.
\]
The first addend is bounded by $\|K\|_{L^p(\W,\g)}$ using H\"older inequality and \eqref{eq:sob_emb}, whereas the second addend is bounded by $\sum_{i=1}^m |K_i|$ using Cauchy--Schwarz inequality and \eqref{eq:trace_ineq}.
\end{proof}

\begin{lemma}
\label{lem:sumc_i}
\beq
\sum_{i=1}^m |c_i| \lesssim \|K\|_{\HMoneMC}
\EqNo{\lesssim}{eq:combine6} \Kbnd.
\label{eq:combine5}
\eeq
\end{lemma}

\begin{proof}
We have
\[
\sum_{i=1}^m |c_i| \lesssim 
\|u|_{\dW}\|_{L^2(\W,\g)} 
\EqNo{\lesssim}{eq:trace_ineq} 
\|u\|_{\dotHnorm} 
=  \|K\|_{\HMoneMC},
\]
where the final equality is the dual norm relation \eqref{eq:dualnorm}.  
\end{proof}

In the following, we will need the following classical elliptic regularity result:

\begin{lemma}
\label{lem:regularity}
Let $\W \subset \mathbf{R}^2$ be a bounded Lipschitz domain, and let $u$ solve the system
\[
\begin{cases}
- \sum_{i,j} a^{ij}\partial_{ij}u + \sum_i b^i\partial_iu + cu=  f &\quad\text{in }\,\, \W,\\
u=0 &\quad\text{in }\,\, \partial\W.
\end{cases}
\]
Let $p \in (1,2)$ and assume that the coefficients satisfy:
\begin{enumerate}[itemsep=0pt,label=(\alph*)]
\item $a^{ij}$ is uniformly continuous, symmetric and uniformly elliptic, with ellipticity constant $\Lambda$.
\item $b_i \in L^q(\W,\euc)$ for $q = 2p/(2-p)$.
\item $c \in L^\infty(\W)$.
\end{enumerate}
Then.
\beq
\| \nabla^2 u \|_{L^p(\W,\euc)} \lesssim C \brk{\|u \|_{L^p(\W,\euc)} + \| f \|_{L^p(\W,\euc)}}.
\label{eq:app_bound_we_need}
\eeq
where $C$ depends on $\W$, $\Lambda, \max_i \|b_i\|_{L^q(\W,\euc)}, \|c \|_{L^\infty(\W)}$, and the maximum of the moduli of continuity of $a^{ij}$.
\end{lemma}

\begin{proof}
If the coefficients $b^i$ in \lemref{lem:regularity} were assumed to be bounded, then this result is classical and can be found, e.g., in \cite{CW98}. The proof in our setting requires only minor modifications of the proof in \cite[Chapters 3.4--3.5]{CW98}: in places where H\"older's inequality $\| f g \|_{L^{p}} \leq C \| f \|_{L^{p}} \| g \|_{L^{\infty}} $ is used, we need to use $\| f g \|_{L^{p}} \leq C \| f \|_{L^{2}} \| g \|_{L^{q}} $, and subsequently use Gagliardo-Nirenberg interpolation to control the $L^2$ term of the function with $L^p$ norm of its derivative. 
\end{proof}

Henceforth, we denote by $u_K^c$ the solution to the Dirichlet boundary value problem
\[
-\Delta_\g u_K^c = K
\qquad
u_K^c|_{\Gamma_0} = 0
\qquad
u_K^c|_{\Gamma_i} = c_i,
\]
that is, $u = u_K^c$ is the solution to \eqref{eq:the_BVP}.
Denote by $v_0^c$ the solution to the auxiliary system
\[
-\Delta_\euc v_0^c = 0
\qquad
v_0^c|_{\Gamma_0} = 0
\qquad
v_0^c|_{\Gamma_i} = c_i,
\]
with the Laplace-Beltrami operator replaced with the Euclidean laplacian. By standard elliptic estimates for harmonic functions,
\beq
\label{eq:v0c}
\|v_0^c\|_{W^{k,p}(\W,\g)} \simeq \|v_0^c\|_{W^{k,p}(\W,\euc)} \le \sum_{i=1}^m |c_i|
\eeq
for every integer $k$ and $p > 1$.
Noting that $u_K^c = u_0^c + u_K^0$, we start by estimating the first term. 

\begin{lemma}
\label{lem:harmonic}
\beq
\|\nabla u_0^c\|_{L^2(\W,\euc)} + \| \nabla^2 u_0^c\|_{L^p(\W,\euc)} \lesssim  \sum_{i=1}^m |c_i|
\EqNo{\lesssim}{eq:combine5} \Kbnd.
\label{eq:combine4}
\eeq
\end{lemma}

\begin{proof}
We start with the following sequence of $L^2$-inequalities and prove the inequality for the first addend.
\beq
\|\nabla u_0^c\|_{L^2(\W,\euc)} 
\EqNo{\lesssim}{eq:equiv_Wkp}  
\|\nabla u_0^c\|_{L^2(\W,\g)} \le 
\|\nabla v_0^c\|_{L^2(\W,\g)} 
\EqNo{\lesssim}{eq:equiv_Wkp} \|\nabla v_0^c\|_{L^2(\W,\euc)} 
\EqNo{\lesssim}{eq:v0c}
\sum_{i=1}^m |c_i|,
\label{eq:combine1}
\eeq
where the second inequality (no constant involved) follows from $u_0^c$ being the minimizer of the Dirichlet functional among all functions satisfying the same boundary conditions.

By the definition of $u_0^c$ and $v_0^c$,
\[
\Delta_\g(u_0^c - v_0^c) = - \Delta_\g v_0^c
\qquad
(u_0^c - v_0^c)|_{\dW} = 0.
\]
Applying \lemref{lem:regularity}:
\[
\begin{aligned}
\|\nabla^2(u_0^c - v_0^c) \|_{L^p(\W,\euc)} &\lesssim 
\|u_0^c - v_0^c\|_{L^p(\W,\euc)} + \|\Delta_\g v_0^c\|_{L^p(\W,\euc)} \\
& \EqNo{\lesssim}{eq:v0c}
\|u_0^c - v_0^c\|_{L^p(\W,\euc)} + \sum_{i=1}^m |c_i|.
\end{aligned}
\]
By the Poincar\'e inequality and the triangle inequality:
\[
\begin{aligned}
\|u_0^c - v_0^c\|_{L^p(\W,\euc)}  &\lesssim \|\nabla (u_0^c - v_0^c)\|_{L^p(\W,\euc)}  
\le 
\|\nabla u_0^c\|_{L^p(\W,\euc)}  + \|\nabla v_0^c\|_{L^p(\W,\euc)} \\
& \EqNo{\lesssim}{eq:v0c}
\|\nabla u_0^c\|_{L^2(\W,\euc)}  + \sum_{i=1}^m |c_i| .
\end{aligned}
\]
Combining with \eqref{eq:combine1},
\[
\|\nabla^2(u_0^c - v_0^c) \|_{L^p(\W,\euc)} \lesssim  \sum_{i=1}^m |c_i|.
\]
One last application of the triangle inequality, using \eqref{eq:v0c} once again gives the desired result.
\end{proof}

We proceed to estimate $u_K^0$:

\begin{lemma}
\beq
\label{eq:combine3}
\|u_K^0\|_{L^p(\W,\g)} \lesssim \|K\|_{\HMoneMC} + \sum_{i=1}^m |c_i|
\stackrel{\eqref{eq:combine5},\eqref{eq:combine6}}{\lesssim}
\Kbnd.
\eeq
\end{lemma}

\begin{proof}
We have
\[
\begin{aligned}
\|u_K^0\|_{L^p(\W,\g)} 
&\lesssim 
\|u_K^0\|_{L^2(\W,\g)} 
\EqNo{\lesssim}{eq:equiv_Wkp} 
\|u_K^0\|_{L^2(\W,\euc)} 
\EqNo{\lesssim}{eq:poincare}
\|\nabla u_K^0\|_{L^2(\W,\euc)} =  
\|\nabla (u_K^c  - u_0^c)\|_{L^2(\W,\euc)} \\
&\lesssim 
\|\nabla u_K^c\|_{L^2(\W,\euc)} +  \|\nabla u_0^c\|_{L^2(\W,\euc)} 
\EqNo{\lesssim}{eq:equiv_Wkp} 
\|\nabla u_K^c\|_{L^2(\W,\g)} +  \|\nabla u_0^c\|_{L^2(\W,\euc)} \\
&\EqNo{\lesssim}{eq:combine4}
\|K\|_{\HMoneMC} + \sum_{i=1}^m |c_i|,
\end{aligned}
\]
where the last inequality hinges also on the dual norm relation \eqref{eq:dualnorm}.
\end{proof}

\begin{lemma}
\label{lem:nab2_uK0}
\beq
\label{eq:combine2}
\|\nabla^2 u_K^0\|_{L^p(\W,\euc)} \lesssim \|u_K^0\|_{L^p(\W,\g)} + \|K\|_{L^p(\W,\g)}
\EqNo{\lesssim}{eq:combine3} \Kbnd.
\eeq
\end{lemma}

\begin{proof}
Writing the interior equation for $u_K^0$ in coordinates:
\[
\g^{ij} \partial_{ij}  u_K^0  - \g^{ij} \Gamma_{ij}^k \partial_k u_K^0  = K,
\]
we apply \lemref{lem:regularity} noting that the matrix $\g^{ij}$ is symmetric, uniformly continuous and uniformly elliptic, whereas
\[
\|\g^{ij} \Gamma_{ij}^k\|_{L^q(\W,\euc)} \lesssim \|\g^{-1}\|_{L^\infty(\W) }\|\nabla \g\|_{L^q(\W,\euc)} 
\]
is controlled by the $W^{2,p}$ norms of $\g$ and its inverse. 
\end{proof}

We finally prove the main result of this appendix.

\begin{proposition}
\label{prop:whatwereallyneed}
Let $u = u_K^c$ be the solution to \eqref{eq:the_BVP}, then 
\[
\|u\|_{L^\infty(\W)} \lesssim \Kbnd.
\]
\end{proposition}

\begin{proof}
Since  $u$ vanishes on $\Gamma_0$, it follows from Sobolev embedding theorem that
\[
\|u\|_{L^\infty(\W)} \lesssim \|\nabla^2 u\|_{L^p(\W,\euc)}
\le  \|\nabla^2 u_K^0\|_{L^p(\W,\euc)} + \|\nabla^2 u_0^c\|_{L^p(\W,\euc)}
\stackrel{\eqref{eq:combine4},\eqref{eq:combine2}}{\lesssim}
\Kbnd.
\]
\end{proof}

\addcontentsline{toc}{section}{References} 
\footnotesize{
\providecommand{\bysame}{\leavevmode\hbox to3em{\hrulefill}\thinspace}
\providecommand{\MR}{\relax\ifhmode\unskip\space\fi MR }
\providecommand{\MRhref}[2]{%
  \href{http://www.ams.org/mathscinet-getitem?mr=#1}{#2}
}
\providecommand{\href}[2]{#2}

}

\end{document}